\documentclass[a4paper, reqno]{amsart}

\usepackage[all]{xy}
\usepackage[english]{babel}
\usepackage[utf8]{inputenc}
\usepackage{amssymb,amsmath,amsthm}
\usepackage{amsfonts}
\usepackage{graphicx}
\usepackage{psfrag}
\usepackage[dvipsnames]{xcolor}
\usepackage[hidelinks]{hyperref}
\usepackage{csquotes}
\usepackage[alphabetic]{amsrefs}
\usepackage{enumerate}
\usepackage{cancel}
\usepackage[normalem]{ulem}
\usepackage{tikz-cd}

\allowdisplaybreaks

\newcommand{\id}{\mathrm{id}}
\newcommand{\Ric}{\mathrm{Ric}}
\newcommand{\Riem}{\mathrm{R}}

\newcommand{\e}{\epsilon}
\renewcommand{\i}{\iota}

\newcommand{\md}{\mathrm{d}}
\newcommand{\on}{\nabla^\s}

\newcommand{\oR}{{\Riem^\s}}

\newcommand{\oRic}{\mathrm{Ric^\s}}

\newcommand{\n}{\nabla}

\renewcommand{\L}{\mathcal{L}}

\renewcommand{\o}{\omega}
\newcommand{\s}{\sigma}

\newcommand{\tr}{\mathrm{tr}}

\renewcommand{\a}{\alpha}
\newcommand{\de}{\delta}
\renewcommand{\b}{\beta}
\renewcommand{\d}{\partial}

\renewcommand{\div}{\mathrm{div}}

\renewcommand{\H}{\mathcal H}

\newcommand{\U}{\mathcal U}
\newcommand{\V}{\mathcal V}

\theoremstyle{plain}
\newtheorem{thm}{Theorem}[section]

\newtheorem{prop}[thm]{Proposition}
\newtheorem{lemma}[thm]{Lemma}
\newtheorem{cor}[thm]{Corollary}
\newtheorem{conj}[thm]{Conjecture}
\theoremstyle{definition}
\newtheorem{definition}[thm]{Definition}
\newtheorem{notation}[thm]{Notation}

\newtheorem{remark}[thm]{Remark}
\newtheorem{example}[thm]{Example}

\newtheorem{assumption}[thm]{Assumption}
\newtheorem{ind_assumption}[thm]{Induction Assumption}

\newcommand{\R}[0]{\mathbb{R}}

\newcommand{\N}[0]{\mathbb{N}}

\makeatletter
\@namedef{subjclassname@1991}{Subject}
\@namedef{subjclassname@2000}{Subject}
\@namedef{subjclassname@2010}{2020 Mathematics Subject Classification}
\makeatother

\title[The asymptotic expansion at the event horizon]{The asymptotic expansion of the spacetime metric at the event horizon}

\author{Klaus Kröncke}
\address{Department of Mathematics, KTH Royal Institute of Technology,
Lindstedtsvägen 25,
11428 Stockholm, Sweden}
\email{kroncke@kth.se}

\author{Oliver Petersen}
\address{Department of Mathematics, Stockholm University,
Albanovägen 28,
10691 Stockholm, Sweden}
\email{oliver.petersen@math.su.se}


\begin{document}

\begin{abstract}
Hawking's local rigidity theorem, proven in the smooth setting by Alexakis-Ionescu-Klainerman, says that the event horizon of any stationary non-extremal black hole is a non-degenerate Killing horizon.
In this paper, we prove that the full asymptotic expansion of any smooth vacuum metric at a non-degenerate Killing horizon is determined by the geometry of the horizon.
This gives a new perspective on the black hole uniqueness conjecture.
In spacetime dimension $4$, we also prove an existence theorem: Given any non-degenerate horizon geometry, Einstein's vacuum equations can be solved to infinite order at the horizon in a unique way (up to isometry).
The latter is a gauge invariant version of Moncrief's classical existence result, without any restriction on the topology of the horizon.
In the real analytic setting, the asymptotic expansion is shown to converge and we get well-posedness of this characteristic Cauchy problem.
\end{abstract}

\maketitle
\tableofcontents
\begin{sloppypar}

\section{Introduction}

Moncrief showed in two remarkable papers, \cites{M1982,M1984}, that the asymptotic expansion of any $4$-dimensional vacuum spacetime metric at a non-degenerate Killing horizon can be written in terms of six  essentially freely specifiable \emph{coordinate dependent} functions on the horizon.
If the functions are real analytic, then Moncrief shows that the asymptotic expansion converges and he obtains a real analytic vacuum metric in a neighborhood of a non-degenerate Killing horizon.
Even though his result is coordinate dependent, this is in stark contrast with spacelike hypersurfaces in vacuum spacetimes and the classical Cauchy problem in general relativity, where it is not in general possible to compute the full asymptotic expansion in terms of freely specifiable functions.
Indeed, at spacelike hypersurfaces, the relativistic constraint equations have first to be solved, which is a highly non-trivial task even on a local coordinate patch.

However, in Moncrief's setup, it does not seem possible to determine whether two obtained vacuum spacetimes are isometric to infinite order at the horizon (or isometric in a neighborhood of the horizon, in case the metric is real analytic).
In other words, the question is: When do two sets of Moncrief data at a non-degenerate Killing horizon give geometrically equivalent vacuum spacetimes near the horizon?
The first main novelty in this paper is a complete answer to this question.
The second main novelty is that we can handle horizons of any topology.
Indeed, using our geometric uniqueness argument, we are able to glue together local solutions on coordinate patches to a global solution, independent of the topology.

Now, the key idea of this paper is to formulate \emph{geometric} data on non-degenerate Killing horizons, in terms of which one can compute the full asymptotic expansion of the spacetime metric in a geometrically canonical way.
In particular, our construction is \emph{gauge independent}, as opposed to Moncrief's approach.
We introduce the following notion:
\begin{definition} \label{def: non-deg Killing data}
\emph{Non-degenerate Killing horizon data} are a smooth Riemannian manifold $(\H, \s)$ equipped with a Killing vector field $V$ of (non-zero) constant length.
\end{definition}

As we will see in Subsection \ref{subsec: induced data}, such initial data are naturally induced on any non-degenerate Killing horizon in any vacuum spacetime.
Our first main result is the following:
\begin{thm} \label{thm: main uniqueness informal}
The asymptotic expansion of a vacuum spacetime metric at a non-degenerate Killing horizon $\H$ is completely determined by a Riemannian metric $\sigma$ and a Killing vector field $V$ of constant length (w.r.t.\ $\sigma$) on the horizon.
\end{thm}

\begin{cor} \label{cor: main uniqueness analytic informal}
If in addition the spacetime metric is real analytic, then it is completely determined by $\s$ and $V$ in an open neighborhood of the horizon.
\end{cor}

\noindent
See Theorem \ref{thm: main uniqueness formal} for the precise formulation. 
Combining this with the following celebrated result of Alexakis-Ionescu-Klainerman (generalizing the work of Hawking in \cites{H1972, HE1973} in the real analytic setting), our result applies to the event horizons of stationary black holes:
\begin{thm}[\cite{AIK2010}] \label{thm: AIK}
The event horizon of any stationary non-extremal black hole (with spherical cross-section and bifurcate horizon) is a non-degenerate Killing horizon.
\end{thm}

\noindent
The assumptions of spherical cross-section and bifurcate horizons were later removed by the second author of this paper in \cite{P2019}*{Thm.\ 1.23} and replaced by compact cross-section and non-degeneracy of the horizon.

We prove the converse to Theorem \ref{thm: main uniqueness informal} in spacetime dimension $4$:
\begin{thm}
 \label{thm: main existence informal}
Given any $3$-dimensional Riemannian manifold $(\H, \s)$, equipped with a Killing vector field $V$ of constant length, there is a unique (up to isometry to infinite order) series expansion solution of Einstein's vacuum equation, with induced data $(\H, \s, V)$ on a non-degenerate Killing horizon.
\end{thm}
\noindent
Note that we have no restriction on the topology of the horizon. 
In Moncrief's approach, one needs a global coordinate system on the horizon.
In our approach, we get rid of this condition by \emph{gluing} Moncrief solutions by applying our uniqueness result, Theorem \ref{thm: main uniqueness informal}.
If the data is real analytic, the asymptotic expansion converges:
\begin{cor} \label{cor: main existence analytic informal}
The real analytic 4-dimensional vacuum spacetimes with a non-degenerate Killing horizon are in one-to-one-correspondence (up to isometry) with the real analytic Riemannian 3-manifolds admitting a Killing vector field of constant length.
\end{cor}

\noindent
See Theorem \ref{thm: main existence formal} for the precise formulation of these statements. 
Corollary \ref{cor: main existence analytic informal} provides the following set of interesting examples, generalizing the Schwarzschild spacetime:

\begin{example} \label{ex: any shape}
Let
\[
	\H := \R \times K,
\]
where $K$ is \emph{any} real analytic Riemannian surface. 
We equip $\H$ with the Riemannian product metric $\s$, implying that $\s$ admits a Killing vector field of constant length.
If $K = S^2$ with the round metric, our result produces the Schwarzschild spacetime (with a certain mass).
If $K$ is any other surface, then our result produces another real analytic vacuum spacetime with a horizon of that shape.
Given two such surfaces $K_1$ and $K_2$, our result implies that the resulting vacuum spacetimes are isometric near the horizon if and only if $K_1$ and $K_2$ are isometric.
In conclusion, our result provides the existence of a horizon of \emph{any real analytic shape}.
\end{example}

\begin{remark}
Example \ref{ex: any shape} shows that our constraint equation is very easy to solve.
Indeed, we just need to find a Riemannian metric with a Killing vector field of constant length\footnote{In fact, it suffices to find a \emph{nowhere vanishing} Killing vector field $V$, for it is a Killing vector field of constant length with respect to the metric
\[
	\hat \s := \frac{\s}{\s(V, V)}.
\]}, which is much simpler than solving the constraint equations for the classical Cauchy problem in general relativity \cite{F-B1952}.
\end{remark}

\begin{remark}
There is an important difference in our geometric formulation of the data on horizons and Moncrief's formulation.
As it turns out, one of Moncrief's six functions, called $\mathring \phi$ in this paper, can always be set to zero.
With all other choices of $\mathring \phi$, one will just produce isometric (up to infinite order) copies of vacuum spacetimes to a choice when $\mathring \phi = 0$.
The correct amount of essentially freely specifiable functions in spacetime dimension $4$ is therefore \emph{five}, as opposed to six functions in Moncrief's approach.
We refer to Remark~\ref{rmk: degrees of freedom} below for the computation of these degrees of freedom.
\end{remark}

One of the fundamental conjectures in mathematical general relativity is the black hole uniqueness conjecture:
\begin{conj}[The black hole uniqueness conjecture] \label{conj: black hole uniqueness}
The only possible domain of outer communication in a $4$-dimensional stationary asymptotically flat vacuum black hole solution to Einstein's vacuum equation is the domain of outer communications in a Kerr black hole.
\end{conj}

\noindent
For the precise formulation of this, we refer to \cite{CC2008}*{Conj.\ 1.2}. 
Let us now denote the induced non-degenerate Killing horizon data (c.f.\ Definition~\ref{def: induced data}) on the event horizon in a subextremal Kerr spacetime by $(\s_{\mathrm{Kerr}}, V_{\mathrm{Kerr}})$ on $\R \times S^2$.
By applying Theorem \ref{thm: main uniqueness informal}, we can replace the technical assumption on the horizon in \cite{IK2009}*{Main~Theorem} (denoted \textbf{T} there) and replace it by the assumption that the horizon geometry coincides with $(\s_{\mathrm{Kerr}}, V_{\mathrm{Kerr}})$:
\begin{thm}[Combining Theorem \ref{thm: main uniqueness informal} with \cite{IK2009}*{Main Theorem}]
Let $M$ be a spacetime satisfying the asymptotic flatness assumption \textbf{AF}  and the smooth bifurcate sphere assumption \textbf{SBS}, respectively, as in \cite{IK2009}*{p. 40-41}. 
If the induced data on the event horizon is $(\s_{\mathrm{Kerr}}, V_{\mathrm{Kerr}})$, then the outer domain of communication in $M$ is locally isometric to a subextremal Kerr spacetime.
\end{thm}

\noindent
It is reasonable to expect that `locally isometric' could be strengthened to `isometric' in this statement.
In that case, in the light of Theorem~\ref{thm: main existence informal}, only $(\s_{\mathrm{Kerr}}, V_{\mathrm{Kerr}})$ would produce the domain of outer communications in the subextremal Kerr spacetime (of a given mass and angular momentum).
By Theorem~\ref{thm: main existence informal}, the non-degenerate case of Conjecture \ref{conj: black hole uniqueness} can therefore be reformulated as follows:
\begin{conj}[The black hole uniqueness conjecture reformulated]
The only data $(\s, V)$ on the event horizon of a $4$-dimensional non-degenerate stationary vacuum black hole solution to Einstein's equation, which gives an asymptotically flat domain of outer communications, is the Kerr black hole data $\left( \s_{\mathrm{Kerr}}, V_{\mathrm{Kerr}} \right)$.
\end{conj}
\noindent
In light of Theorem \ref{thm: main uniqueness informal} and this formulation of the black hole uniqueness conjecture, one is lead to ask whether the asymptotic flatness of the spacetime can be read off from the asymptotic expansion of the metric at the horizon.
This is a topic for future research.

\subsection{Induced non-degenerate Killing horizon data} \label{subsec: induced data}
Let us now explain how non-degenerate Killing horizon data, in the sense of Definition \ref{def: non-deg Killing data}, are naturally induced on any non-degenerate Killing horizon in any vacuum spacetime.
For this, let $(M, g)$ be a smooth spacetime, i.e.\ a time-oriented connected Lorentzian manifold of dimension $n+1 \geq 2$, with an embedded smooth lightlike hypersurface 
\[
	\iota: \H \hookrightarrow M.
\]
We will often suppress the embedding and simply write $\H \subset M$.

\begin{definition}\label{def: Hawking vector field}
A smooth Killing vector field $W$ on $M$, such that
\[
	W|_\H
\]
is nowhere vanishing, lightlike and tangent to $\H$, is called a \emph{horizon Killing vector field} on $M$ with respect to $\H$.
A smooth lightlike hypersurface $\H$, with respect to which there is a horizon Killing vector field on $M$, is called a \emph{Killing horizon} in $M$.
\end{definition}
\noindent
Assuming the curvature condition
\begin{equation} \label{eq: curvature condition}
	\Ric(W|_\H, X) = 0
\end{equation}
for all $X \in T\H$, the existence of a horizon Killing vector field implies that the surface gravity is constant, i.e.\ that there is a constant $\kappa \in \R$, called \emph{surface gravity}, such that
\begin{equation} \label{eq: surface gravity}
	\n_WW|_\H 
		= \kappa W|_\H.
\end{equation}
For a proof of this classical fact, see e.g.~\cite{PV2021}*{Rem.\ 1.9, Lem.\ B.1}.

\begin{definition}
Assume that \eqref{eq: curvature condition} is satisfied.
We say that $\H$ is a \emph{non-degenerate Killing horizon}, with respect to $W$, if $\kappa \neq 0$.
\end{definition}

\noindent
By replacing $W$ with $-W$ if necessary, we may always assume that $\kappa > 0$:

\begin{assumption}
We assume throughout this paper that $\H \subset M$ is a smooth non-degenerate Killing horizon with $\kappa > 0$.
\end{assumption}
\noindent
Let us now explain how $\s$ and $V$ are constructed from $g$ and $W$.
For this, we first need to recall some properties of Killing horizons.
Note that
\begin{align*}
	g(\n_XW, Y)|_\H 
		&= \frac12 \L_Wg(X, Y)|_\H - \frac12g(W, [X, Y])|_\H = 0
\end{align*}
for all $X, Y \in T\H$, which implies that $\n_XW|_\H$ is normal to the lightlike hypersurface $\H$ and must hence be a multiple of the lightlike direction $W|_\H$ at each point.
In other words, there is a unique smooth one-form $\o$ on $\H$, such that
\begin{equation} \label{eq: def omega}
	\n_XW|_\H = \omega(X) W|_\H.
\end{equation}
It immediately follows that
\begin{equation} \label{eq: omega invariance}
	\L_{W|_\H} \o = 0.
\end{equation}
Now, define the $(0,2)$-tensor $\s$ at any $p \in \H$ as
\begin{equation} \label{eq: sigma}
	\sigma(X, Y) := g(X, Y) + \o(X)\o(Y)
\end{equation}
for all $X, Y \in T_p\H$.
Since 
\[
	\o(W|_\H) = \kappa > 0
\]
by assumption, we note that $\s$ is positive definite, i.e.\ $\sigma$ is a \emph{Riemannian metric} on $\H$.
Let us also introduce the notation
\begin{equation} \label{eq: Riem Killling VF}
	V := W|_\H.
\end{equation}

\begin{definition} \label{def: induced data}
We call $(\H, \sigma, V)$ as defined in \eqref{eq: sigma} and \eqref{eq: Riem Killling VF} the \emph{induced non-degenerate Killing horizon data}.
\end{definition}

\noindent
We have thus shown that each vacuum spacetime with a smooth non-degenerate Killing horizon gives induced data $(\H, \sigma, V)$ in a geometric way.
Note the following two important properties of $\s$ and $V$:
\begin{itemize}
\item Firstly, we have
\[
	\L_V \s 
		= \L_W g|_\H + \L_V\o \otimes \o + \o \otimes \L_V \o = 0,
\]
i.e.\ $V$ is a Killing vector field with respect to $\s$.
\item Secondly, note that
\[
	\s(V, V) = \kappa^2
\]
is constant, i.e.\ the length of $V$ with respect to $\s$ is constant.
\end{itemize}

\subsection{Main result}

The main novelty in this paper is that two smooth vacuum spacetimes with non-degenerate Killing horizons are isometric to infinite order at the horizon, in a way which respects the (Lorentzian) Killing vector fields, precisely when the induced data is the same:

\begin{thm}[Uniqueness of the asymptotic expansion] \label{thm: main uniqueness formal}
Let $(M, g)$ and $(\hat M, \hat g)$ be smooth spacetimes.
Assume that
\begin{center}
\begin{tikzcd}
\H \arrow[hookrightarrow]{dr}{\hat \iota} \arrow[hookrightarrow]{r}{\iota} & M \\
 & \hat M
\end{tikzcd}
\end{center}
are two embeddings of $\H$ as non-degenerate Killing horizons with horizon Killing vector fields $W$ and $\hat W$, respectively, such that $(\H, \s, V)$ is the induced data in both cases.
If furthermore
\begin{align*}
	\n^k \left( \Ric_g \right)|_{\iota(\H)}
		&= 0, \\
	\hat \n^k \left( \Ric_{\hat g} \right)|_{\hat \iota(\hat \H)}
		&= 0
\end{align*}
for all $k \in \N_0$, then there are open neighborhoods $\i(\H) \subset \U \subset M$ and $\hat \i (\hat \H) \subset \hat \U \subset \hat M$ and a diffeomorphism
\[
	\Phi: (\U, g|_\U) \to (\hat \U, g|_{\hat \U})
\]
such that
\begin{align*}
	\Phi \circ \iota 
		&= \hat \iota, \\
	\Phi^* \hat W 
		&= W
\end{align*}
and
\begin{align*}
	\n^k \left( \Phi^*\hat g - g\right)|_{\i(\H)}
		&= 0
\end{align*}
for all $k \in \N_0$.
If in addition $g$ and $\hat g$ are real analytic, then $\U$ can be chosen such that
\begin{align*}
	\Phi^*\hat g|_\U
		&= g|_\U.
\end{align*}
\end{thm}
\noindent
For real analytic solutions, this in particular means that the spacetime metrics are isometric in neighborhoods of the horizons.

Our second main result is the converse statement for $3$-dimensional data:

\begin{thm}[Existence of the asymptotic expansion] \label{thm: main existence formal}
Assume that $(\H, \s)$ is a smooth $3$-dimensional Riemannian manifold, equipped with a Killing vector field $V$ of constant length, i.e.\
\[
	\L_V \sigma = 0
\]
and 
\[
	\sigma(V, V)
\]
is constant. 
Then there is a smooth $4$-dimensional spacetime $(M, g)$ and an embedding
\begin{center}
\begin{tikzcd}
	\H \arrow[hookrightarrow]{r}{\iota} & M
\end{tikzcd}
\end{center}
such that $\H$ is a non-degenerate Killing horizon with respect to a horizon Killing vector field $W$ on $M$ and
\begin{itemize}
	\item $\n^k \Ric_g|_{\i(\H)} = 0$ for all $k \in \N_0$,
	\item $(\H, \sigma, V)$ is the induced data as in Definition \ref{def: induced data}.
\end{itemize}
If in addition $\s$ is real analytic, then $\iota$ is real analytic and
\[
	\Ric_g = 0
\]
in an open neighborhood of $\iota(\H)$.
\end{thm}

\noindent
In other words, given any data, there is a spacetime metric which is vacuum to infinite order at the horizon.
For real analytic data, we get a solution to Einstein's vacuum equation in an open neighborhood of the horizon.

\begin{remark}
It is conceivable that Theorem~\ref{thm: main existence formal} extends verbatim to higher dimensions. 
Indeed, all novelties of this paper extend to higher dimensions. 
We only restrict to $3 + 1$ dimensions in Theorem~\ref{thm: main existence formal} because Moncrief's existence result in \cite{M1982} is proven in $3 + 1$ dimensions, c.f.~Theorem~\ref{thm: Moncrief} below.
\end{remark}

\subsubsection{Motivating the data} \label{sec: sigma g omega}
Let us give a brief motivation why the data $(\s, V)$ are correct for this characteristic Cauchy problem, by showing how the spacetime metric at the horizon and the one-form $\o$ are constructed from $(\s, V)$.
Given a data set $(\H, \s, V)$, note that the quadratic form
\[
	g(X, Y) := \s(X, Y) - \frac{\s(X, V) \s(Y, V)}{\s(V, V)}
\]
for any $X, Y \in T_p\H$ is a \emph{lightlike} metric at $\H$, i.e.\ $g(X, X) \geq 0$ and
\[
	g(X, X) = 0 \Leftrightarrow X \text{ is parallel to } V.
\]
Moreover, defining 
\[
	\o(X) := \frac{\s(X, V)}{\sqrt{\s(V, V)}},
\]
we have
\[
	\s(X, Y) = g(X, Y) + \o(X) \o(Y).
\]
The quadratic form $g$ will be the induced lightlike metric at the horizon in the resulting spacetime, and $\o$ will be the one-form defined by \eqref{eq: def omega}.

\subsubsection{Degrees of freedom}

Let us now count the (local) degrees of freedom in $4$-dimensional vacuum spacetimes:

\begin{remark} \label{rmk: degrees of freedom}
In a $4$-dimensional spacetime, the horizon will be $3$-dimensional with local coordinates $x_1, x_2, x_3$.
Without loss of generality, let us choose $V := \d_{x_1}$ and write 
\[
	\s 
		= \sum_{i,j = 1} \s_{ij}\md x_i \otimes \md x_j.
\]
The fact that $\s_{11} = \s(V, V)$ is constant is equivalent to choosing $\s_{11}$ to be a non-zero constant.
Since the scaling of $V$ does not matter, we may without loss of generality choose $\s_{11} = 1$.
The fact that
\[
	\L_V \s = 0
\]
is just to say that $\s_{ij}$ are all independent of the coordinate $x_1$.
Moreover, of course $\s_{ij} = \s_{ji}$.
This reduces the degrees of freedom precisely to choosing the \emph{five} functions $\s_{12}, \s_{13},\s_{23}, \s_{22}, \s_{33}$, dependent only on $x_2, x_3$, in such a way that $\s$ is a Riemannian metric.
As explained above, this is one function less than in Moncrief's approach.
\end{remark}

\subsubsection{Relation to earlier results}

Our asymptotic expansion of the spacetime metric at the event horizon is completely analogous to the celebrated Fefferman-Graham expansion of (asymptotically) Poincar\'e-Einstein metrics at conformal infinity \cites{FG1984,FG2007}.
The method of Fefferman and Graham is their \emph{ambient construction}, which corresponds to expanding a Lorentzian vacuum spacetime metric around a generalization of the light cone in the Minkowski spacetime.
Our result is very similar in spirit, with the light cone replaced by the event horizon of a stationary black hole.
However, their result is mathematically unrelated and cannot be applied in the event horizon setting.
Parallel to our work, there has also been interesting recent developments by Holzegel-Shao \cite{HS2022}, which is based on a Fefferman-Graham type expansion at the conformal boundary of asymptotically Anti-de Sitter spacetimes.

It is a classical topic in general relativity to show that horizons in vacuum spacetimes are Killing horizons, under weak assumptions.
Our results here apply to all such results, if the Killing horizon is non-degenerate.
As mentioned above, this was initiated by Hawking as the main novelty in his black hole uniqueness theorem \cites{H1972, HE1973}.
The analyticity assumptions in Hawking's theorem was dropped for stationary black holes by Alexakis-Ionescu-Klainerman in \cites{AIK2010}.
See also \cites{IK2009, IK2013} and the related work in real analytic spacetimes by Chru\'{s}ciel-Costa \cite{CC2008}, Moncrief-Isenberg \cite{MI2008} and Hollands-Ishibashi-Wald \cite{HIW2007}.

As a parallel to this, it is known that any non-degenerate compact Cauchy horizon in a smooth vacuum spacetime is a non-degenerate Killing horizon, by combining the result of the second author in \cite{P2019} with the result of Bustamente-Reiris \cite{BR2021} (see also \cite{GM2021} for a streamlined proof of the latter).
These are in turn using the results in \cites{P2018, PR2018, L2015, M2015} and the pioneering work in the analytic setting by Moncrief-Isenberg \cites{MI1983, MI2018}.

Let us finally mention that Geroch and Hartle in \cite{GH1982} find all exact \emph{static} and \emph{axisymmetric} $4$-dimensional black holes distorted by a external matter distribution. 
Their stronger assumptions allow them to discuss the global structure, in particular the asymptotics to infinity.

\vspace{3mm}
\noindent\textbf{Acknowledgements.}
We would like to thank Vincent Moncrief and Piotr Chru\'sciel for very valuable conversations.
The first author is grateful for the support of the Deutsche Forschungsgemeinschaft (DFG, German Research Foundation) through the priority programme 2026 \textit{Geometry at Infinity}. 
The second author is grateful for the support of the DFG – GZ: LI 3638/1-1, AOBJ: 660735 and the Stanford Mathematics Research Center.

\section{Uniqueness of the expansion} \label{sec: as exp}

The key in our arguments is to show that the full asymptotic expansion of the spacetime metric $g$ at the horizon is given \emph{geometrically} by the data $(\H, \s, V)$, assuming the Ricci curvature vanishes to infinite order at the horizon:

\begin{assumption} \label{ass: main}
Assume that $M$ is a smooth spacetime of dimension $n+1 \geq 2$, with a smooth Killing horizon $\H \subset M$, with horizon Killing vector field $W$, such that
\begin{equation} \label{eq: Ricci vanishing}
	\n^k\Ric|_\H = 0
\end{equation}
for all $k \in \N_0$.
\end{assumption}
We assume throughout this section that Assumption \ref{ass: main} is satisfied.

\subsection{A geometric gauge} \label{subsec: null time function}

Let $(\H, \s, V)$ denote the induced data (in the sense of Definition \ref{def: induced data}).
Let
\[
	V^\perp \subset T\H
\]
denote the smooth vector bundle of tangent vectors which are orthogonal to $V$, with respect to $\s$.
For notational convenience, we do not write out the embedding $\iota$ in this section and instead simply write
\[
	\H \subset M, \quad V = W|_\H,
\]
etc.

\begin{definition}
The unique lightlike smooth vector field $L$ along $\H$ which satisfies
\[
	g(L, V) = 1, \quad g(L, X) = 0,
\]
for all $X \in V^\perp$ is called the \emph{canonical transversal vector field} of $\H$.
\end{definition}

Note that indeed $L$ is not tangent to $\H$, since 
\[
	0 \neq g(L, V) = g(L, W|_\H).
\]
We think of $L$ as a natural replacement for the unit normal along a spacelike or timelike hypersurface.

\begin{prop}\label{prop: the d_t vector field}
There is a unique nowhere vanishing lightlike smooth (real analytic, if the metric is real analytic) vector field $\d_t$ on an open neighborhood $\U \supset \H$ satisfying
\begin{itemize}
	\item $\n_{\d_t}\d_t = 0$,
	\item $g(\d_t|_\H, V) = 1$,
	\item $g(\d_t|_\H, X) = 0$ for all $X \in V^\perp$, 
	\item all integral curves of $\d_t$ intersect $\H$ precisely once.
\end{itemize}
Moreover, it follows that
\begin{equation} \label{eq: dt and W commuting}
	[W, \d_t] 
		= 0.
\end{equation}
\end{prop}
\begin{proof}
We construct $\d_t$ as the unique solution to
\begin{align*}
	\n_{\d_t}\d_t
		&= 0, \\
	\d_t|_\H
		&= L,
\end{align*}
on a neighborhood $\U \supset \H$.
It follows that
\[
	\d_tg(\d_t, \d_t) = 0,
\]
which together with
\[
	g(\d_t, \d_t)|_\H = g(L, L)|_\H = 0
\]
implies that $\d_t$ is lightlike.
It remains to check that $[W, \d_t] = 0$.
For this, we first show that $[W, \d_t]|_\H = 0$.
Note first that
\begin{align*}
	g([W, \d_t]|_\H, L)
		&= g(\n_W \d_t|_\H, L) - g(\n_{\d_t} W|_\H, L) \\
		&= \frac12 W g(L, L) - \frac12 \L_Wg(L, L) \\
		&= 0.
\end{align*}
Note now that $g(\d_t|_\H, X) = g(L, X) = \frac{\o(X)}\kappa$, for all $X \in T\H$, since $g(L, X) = 0$ for all $X \perp_\s V$, i.e.~for all $X \in \ker(\o)$, and $g(L, V) = 1 = \frac{\o(V)}{\kappa}$.
Using \eqref{eq: omega invariance}, we compute
\begin{align*}
	g([W, \d_t]|_\H, X) 
		&= Wg(\d_t, X)|_\H - \L_Wg(\d_t, X)|_\H - g(\d_t, [W, X])|_\H \\
		&= Wg(\d_t, X)|_\H  - g(\d_t, [W, X])|_\H \\
		&= \frac{W \o(X)}\kappa - \frac{\o([W, X])}\kappa \\
		&= \frac{\L_W\o(X)}\kappa \\
		&= 0.
\end{align*}
It follows that $[W, \d_t]|_\H = 0$.
Now, since $\L_Wg = 0$, we note that
\[
	0 = \L_W\left(\n_{\d_t}\d_t \right) = \n_{[W, \d_t]}\d_t + \n_{\d_t}[W, \d_t], 
\]
which is a linear first order ODE.
It thus follows that $[W, \d_t] = 0$ as claimed.
Shrinking $\U$ if necessary, we can ensure that each integral curve of $\d_t$ intersects $\H$ precisely once. 
\end{proof}

\begin{prop}[The null time function] \label{prop: null time function}
Let Assumption \ref{ass: main} be satisfied.
There is a unique smooth (real analytic, if the metric is real analytic) function
\[
	t: \U \to \R,
\]
such that
\begin{align*}
	\md t(\d_t) 
		&= 1, \\
	t^{-1}(0)
		&= \H \cap \U.
\end{align*}
In particular, $\md t \neq 0$ everywhere on $\U$.
\end{prop}
\begin{proof}
Construct the function $t$ as the eigentime of the integral curves of $\d_t$, starting at $\H$.
Then $t$ is smooth and satisfies the assertion.
\end{proof}

\begin{definition}
We call $t: \U \to \R$ the \emph{null time function}.
\end{definition}

\subsection{The expansion}
Since the null time function $t$ is constructed geometrically, it is very natural to study the asymptotic expansion of the metric $g$ in terms of $t$, i.e.\ we formally write
\[
	g \sim \sum_{m = 0}^\infty \L_t^mg|_\H \frac{t^m}{m!},
\]
and iteratively compute $\L_t^m g|_\H := \underbrace{\L_{\d_t}\hdots \L_{\d_t}}_{m \text{ times}}g|_\H$ in terms of the data $(\s, V)$.
We will also use the notation
\[
	\n_t := \n_{\d_t}.
\]

\begin{remark} \label{rmk: g at H}
Proposition \ref{prop: the d_t vector field} implies that
\begin{align*}
	g(\d_t, V)|_\H 
		&= 1, \\
	g(X, Y)|_\H
		&= \s(X, Y), \\
	g(\d_t, \d_t)|_\H 
		&= g(\d_t, X)|_\H
		= g(V, V)|_\H
		= g(X, V)|_\H
		= 0, \label{eq: g at horizon}
\end{align*}
for all $X, Y \in V^\perp$.
Consequently, since $\d_t|_\H$ is determined by the data $(\s, V)$ and the embedding $\iota$ (which is suppressed here), we conclude that $g|_\H$ is determined by $(\s, V)$.
\end{remark}

We note that many components of $\L_t^mg$ automatically vanish:

\begin{lemma} \label{le: d t components}
By construction of $\d_t$, we have
\begin{align*}
	\L_t^mg(\d_t, \cdot) 
		&= 0,
\end{align*}
for all $m \in \N$.
\end{lemma}
\begin{proof}
For $m = 1$, we note that for any $X$ with $[\d_t, X] = 0$, we have
\[
	\L_tg(\d_t, X) = \d_tg(\d_t, X) = g(\n_t \d_t, X) + g(\d_t, \n_t X) = - \frac12 Xg(\d_t, \d_t) = 0.
\]
The general statement then follows by induction, by noting that
\[
	\L_t^mg(\d_t, X)
		= \d_t \L_t^{m-1}g(\d_t, X). \qedhere
\]
\end{proof}

For the remaining components of $\L_t^mg|_\H$, we will prove the following theorem:

\begin{thm} \label{thm: uniqueness in gauge}
Let Assumption \ref{ass: main} be satisfied.
Then there are unique (non-linear) differential operators $Q_m$ on $\H$ for $m \in \N$, such that
\[
	\L_t^mg(X, Y)|_\H = Q_m(\sigma, V)(X, Y),
\]
for all $m \in \N$ and all $X, Y \in T\H$.
Moreover, we have
\begin{equation} \label{eq: Q geometric}
	Q_m(\phi^*\s, \phi^*V) = \phi^*Q_m(\s,V),
\end{equation}
for all diffeomorphisms $\phi: \H \to \H$ and the $Q_m(\s, V)$'s are real analytic if $\s$ and $V$ are real analytic.
\end{thm}

In Subsection \ref{subsec: uniqueness}, we show that Theorem \ref{thm: uniqueness in gauge} implies Theorem \ref{thm: main uniqueness formal}.
The condition \eqref{eq: Q geometric} is very natural, it says that the differential operators $Q_m$ are \emph{diffeomorphism invariant}.
Theorem \ref{thm: uniqueness in gauge} would obviously not be true without assuming \eqref{eq: Ricci vanishing}, i.e.\ that the Ricci curvature vanishes to infinite order.
The rest of this section will be devoted to prove Theorem \ref{thm: uniqueness in gauge}, i.e.\ to study the remaining components of $\L_t^mg|_{\H}$.

\subsection{The first derivative}

We start by computing $Q_1(\s, V)$ in Theorem \ref{thm: uniqueness in gauge}.
It turns out that its components are given by simple explicit formulas.
Let $\on$ denote the Levi-Civita connection with respect to $\s$ and let $\oR$ and $\oRic$ denote the curvature tensor and the Ricci curvature of $\s$, respectively.

\begin{prop} \label{prop: first derivative}
In terms of the null time function $t$, we have
\begin{align*}
	\L_tg(V, V)|_\H
		&= - 2\kappa, \\
	\L_tg(V, X)|_\H
		&= 0, \\
	\L_tg(X, Y)|_\H
		&= \frac1\kappa \left(\oRic(X, Y) + \frac1{\kappa^2} \s(\on_X V, \on_Y V)\right)
\end{align*}
for all $X, Y \in V^\perp$, where $\kappa$ is the surface gravity with respect to $W$, defined in \eqref{eq: surface gravity}.
Moreover, we have
\begin{align}
	\n_tW|_\H 
		&= - \kappa \d_t|_\H. \label{eq: nabla t W}
\end{align}
\end{prop}

\begin{remark} \label{rmk: first derivative} 
We conclude that
\begin{align*}
	Q_1(\s, V)(V, V)
		&= - 2\kappa, \\
	Q_1(\s, V)(V, X)
		&= 0, \\
	Q_1(\s, V)(X, Y)
		&= \frac1\kappa \left(\oRic(X, Y) + \frac1{\kappa^2} \s(\on_X V, \on_Y V)\right),
\end{align*}
which is a diffeomorphism invariant differential operator in $(\s, V)$.
Since we use the convention that $\kappa > 0$, equation \eqref{eq: sigma} implies that
\[
	\kappa = \o(V) = \sqrt{\s(V, V)}.
\]
This proves the assertion in Theorem \ref{thm: uniqueness in gauge} for $m = 1$.
\end{remark}

In order to prove Proposition \ref{prop: first derivative}, we begin with the following lemma:

\begin{lemma} \label{le: sigma ricci computation}
We have
\[
	\sum_{i,j = 2}^n g^{ij} \Riem(X, e_i, e_j, Y)|_\H = \oRic(X, Y) + \frac1{\kappa^2} \s(\on_X V, \on_Y V),
\]
for any $X, Y \in V^\perp$, where $\{e_2, \hdots, e_n\}$ is a basis for $V^\perp \subset T\H$ and $g^{ij}$ denotes the inverse of $g_{ij} := g(e_i, e_j)$.
\end{lemma}
\begin{proof}
We need to compare $\n$ and $\on$.
Since 
\[
	\sigma(X, Y) = g(X, Y)|_\H,
\]
for all $X \in V^\perp$ and $Y \in T\H$, the Koszul formula implies for all $X, Y, Z \in V^\perp$ that
\begin{align*}
	2 \s(\on_X Y, Z)
		&= X \s(Y, Z) + Y \s(X, Z) - Z \s(X, Y) \\*
		&\qquad + \sigma([X, Y], Z) - \sigma([X, Z], Y) - \s([Y, Z], X) \\
		&= 2g(\n_XY, Z)|_\H.
\end{align*}
This implies for all $X, Y \in V^\perp$ that
\begin{align*}
	\on_XY 
		&= \sum_{i,j = 2}^n \s(\on_XY, e_i)\s^{ij}e_j + \frac1{\kappa^2}\s(\on_XY, V) V \\
		&= \sum_{i,j = 2}^n g(\n_XY, e_i)g^{ij}e_j|_\H + \frac1{\kappa^2} \s(\on_XY, V) V \\
		&= \n_XY|_\H - g(\n_XY, \d_t)|_\H V + \frac1{\kappa^2} \s(\on_XY, V) V,
\end{align*}
where we have used that $g(\n_X Y, V)|_\H = - g(Y, \n_X V)|_\H = - \frac12 \L_W g(X, Y)|_\H = 0$.
The Koszul formula further implies for $X, Y \in V^\perp$ that
\begin{align*}
	2 \s(\on_V X, Y)
		&= V \s(X, Y) + X \s(V, Y) - Y \s(V, X) \\
		&\qquad + \sigma([V, X], Y) - \sigma([V, Y], X) - \s([X, Y], V) \\
		&= 2g(\n_VX, Y)|_\H - \sigma([X, Y], V).
\end{align*}
We may now compute for all $X, Y \in V^\perp$ that
\begin{align*}
	 \Riem&(X, Y, Y, X)|_\H \\
	 	&= g(\n_X \n_YY - \n_Y \n_XY - \n_{[X, Y]}Y, X)|_\H \\
	 	&= g(\n_X( \on_YY + g(\n_YY, \d_t) V - \frac1{\kappa^2} \s(\on_YY, V) V), X)|_\H \\*
	 	&\qquad - g(\n_Y( \on_XY + g(\n_XY, \d_t) V - \frac1{\kappa^2}\s(\on_XY, V) V), X)|_\H \\*
		&\qquad - \sum_{i,j = 2}^n \s([X, Y], e_i)\s^{ij} g(\n_{e_j}Y, X)|_\H - \frac1{\kappa^2} \s([X, Y], V) g(\n_V Y, X)|_\H \\
		&= g(\n_X \on_YY, X)|_\H - g(\n_Y \on_XY, X)|_\H \\*
		&\qquad - \sum_{i,j = 2}^n \s([X, Y], e_i)\s^{ij} g(\on_{e_j}Y, X)|_\H \\*
		&\qquad - \frac1{\kappa^2}\s([X, Y], V) g(\n_V Y, X)|_\H \\ 
		&= \s(\on_X \on_YY, X) - \s(\on_Y \on_XY, X) - \s(\on_{[X, Y]}Y, X) \\*
		&\qquad + \frac1{\kappa^2} \s([X, Y], V) \s(\on_V Y, X) - \frac1{\kappa^2}\s([X, Y], V) g(\n_V Y, X)|_\H \\
		&= \oR(X, Y, Y, X) + \frac1{2\kappa^2} \s([X, Y], V)^2 \\
		&= \oR(X, Y, Y, X) + \frac1{2\kappa^2}\left( \sigma(\on_YV, X) -  \sigma(\on_XV, Y)\right)^2 \\
		&= \oR(X, Y, Y, X) + \frac2{\kappa^2} \sigma(\on_XV, Y)^2,
\end{align*}
where we have used that $\L_V\s = 0$.
It follows that
\begin{align*}
	 \sum_{i,j = 2}^n g^{ij} \Riem(X, e_i, e_j, X)|_\H
	 	&=  \sum_{i,j = 2}^n \s^{ij} \oR(X, e_i, e_j, X) + \frac2{\kappa^2} \s(\on_X V, \on_X V) \\
	 	&= \oRic(X, X) - \frac1{\kappa^2} \oR(X, V, V, X) + \frac2{\kappa^2} \s(\on_X V, \on_X V).
\end{align*}
In order to compute $\oR(X, V, V, X)$, we first note that for all $X \in T\H$, we have
\begin{align*}
	\s(\on_VV, X)
		&= \L_V \s(V, X) - \s(\on_X V, V) \\
		&= - \frac12 X \s(V, V) \\
		&= 0,
\end{align*}
from which we conclude that
\[
	\on_V V = 0.
\]
We may thus compute for all $X \in V^\perp$ that 
\begin{align*}
	\oR(X, V, V, X) 
		&= \s(\on_X \on_V V - \on_V \on_X V - \on_{[X, V]} V, X) \nonumber \\
		&= - V\s(\on_XV, X) + \s(\on_XV, \on_V X) \\
		&\qquad - \L_V\s([X, V], X) + \s(\on_X V, [X, V]) \\
		&= - \frac12 V \L_V \s(X, X) + \s(\on_X V, \on_X V) \\
		&= \s(\on_X V, \on_X V),
\end{align*}
which completes the proof.
\end{proof}

\begin{proof}[Proof of Proposition \ref{prop: first derivative}]
Using that $g(\partial_t,W)|_\H = 1$  and $\nabla_WW|_\H= \kappa W|_\H$, we get
\[
	\L_tg(W,W)|_\H 
		=2g(\nabla_W\partial_t,W)|_\H
		=-2g(\partial_t,\nabla_WW)|_\H
		=-2g(\partial_t,\kappa W)|_\H
		=-2\kappa.
\]
Let $X$ be a smooth vector field, such that $X|_\H \in C^\infty(V^\perp)$ and $[X, \d_t] = 0$.
We compute
\begin{align*}
	\L_tg(W,X)|_\H
		&=g(\nabla_W\partial_t,X)|_\H + g(W,\nabla_X\partial_t)|_\H \\
		&=-g(\partial_t,\nabla_WX)|_\H - g(\nabla_XW,\partial_t)|_\H \\
		&= - g([W,X],\partial_t)|_\H -2g(\partial_t,\nabla_XW)|_\H \\
		&= - \o([W, X]|_\H) g( V,\d_t)|_\H -2 \o(X|_\H)g(\d_t, W)|_\H \\
		&= - \o([W, X]|_\H) \\
		&= - W|_\H \o(X) + \L_{W|_\H} \o (X|_\H) \\*
		&= 0,
\end{align*}
where we in the last step used \eqref{eq: omega invariance}.
Let us now prove \eqref{eq: nabla t W}, using Lemma \ref{le: d t components}.
For any $X \in C^\infty(V^\perp)$, we compute
\begin{align*}
0
	&=\L_Wg(X,\partial_t)|_\H
	=g(\nabla_XW,\partial_t)|_\H + g(X,\nabla_tW)|_\H
	=g(X,\nabla_tW)|_\H, \\
0
	&=\L_Wg(\partial_t,\partial_t)|_\H
	=2g(\nabla_t W,\partial_t)|_\H,\\
0
	&=\L_Wg(W,\partial_t)|_\H
	=g(\nabla_WW,\partial_t)|_\H +g(W,\nabla_tW)|_\H 
	= \kappa + g(W,\nabla_tW)|_\H.
\end{align*}
By Remark \ref{rmk: g at H}, we conclude that $\nabla_tW =- \kappa \partial_t$. 
We use this in the final computation of $\L_Wg(X, Y)|_\H$ for $X, Y \in C^\infty(V^\perp)$.
Since $[W,\partial_t]=0$, we note that
\[
\L_W(\L_tg)=\L_t(\L_Wg)+\L_{[W,\partial_t]}g=0.
\]
We compute
\begin{align*}
	0
		&=\L_W(\L_tg)(X,Y) \\
		&=W\left(\L_tg(X,Y)\right) - \L_tg([W,X],Y) - \L_tg(X,[W,Y]) \\
		&=W\left(g(\nabla_X\partial_t,Y)+g(X,\nabla_Y\partial_t)\right)
-g(\nabla_{[W,X]}\partial_t,Y)\\*
		&\qquad - g([W,X],\nabla_Y\partial_t) - g(\nabla_{[W,Y]}\partial_t,X) - g([W,Y],\nabla_X\partial_t)\\
		& =g(\nabla_W\nabla_X\partial_t-\nabla_{[W,X]}\partial_t,Y)+g(X,\nabla_W\nabla_Y\partial_t-\nabla_{[W,Y]}\partial_t)\\*
		&\qquad +g(\nabla_X\partial_t,\nabla_WY)+g(\nabla_WX,\nabla_Y\partial_t)\\*
		&\qquad -g([W,X],\nabla_Y\partial_t)-g([W,Y],\nabla_X\partial_t)\\
		&= \Riem(W, X, \d_t, Y) +g(\nabla_X\nabla_W\partial_t,Y)+  \Riem(W, Y, \d_t, X) + g(\nabla_Y\nabla_W\partial_t,X)\\
		&\qquad + g(\nabla_X\partial_t,\nabla_YW) + g(\nabla_Y\partial_t,\nabla_XW).
\end{align*}
Evaluating this at $\H$, with $X, Y \in C^\infty(V^\perp)$, using \eqref{eq: nabla t W}, we get
\begin{align*}
	0
		&= \Riem(W, X, \d_t, Y)|_\H - \kappa g(\n_X \d_t, Y)|_\H \\
		&\qquad + \Riem(W, Y, \d_t, X)|_\H - \kappa g(\nabla_Y\partial_t,X)|_\H \\
		&= \Riem(W, X, \d_t, Y)|_\H + \Riem(W, Y, \d_t, X)|_\H - \kappa \L_t g(X, Y)|_\H \\
		&= - \Ric(X, Y)|_\H - \sum_{i,j = 2}^n g^{ij}\Riem(e_i, X, e_j, Y)|_\H - \kappa \L_t g(X, Y)|_\H,
\end{align*}
where $e_2, \hdots, e_n$ is a basis for $V^\perp$.
The proof is now completed by applying Lemma \ref{le: sigma ricci computation} and recalling that $\Ric|_\H = 0$.
\end{proof}

The following corollary will be useful when computing the higher derivatives:
\begin{cor} \label{cor: the full t derivative}
For all $X,Y\in C^{\infty}(V^\perp)$, we have
\begin{align*}
	g(\nabla_X\partial_t,Y)|_\H
		=\frac{1}{2\kappa}\left(\oRic(X, Y) + \frac1{\kappa^2} \s(\on_X V, \on_Y V)+\md\omega(X,Y)\right).
\end{align*}
\end{cor}
\begin{proof}
We compute
\begin{align*}
\md\omega(X,Y)
	&= X\o(Y) - Y\o(X) -\omega([X,Y]) \\
	&= - \o([X, Y]) \\
	&= - g([X, Y], \d_t)|_\H \o(V) \\
	&= - \kappa g([X, Y], \d_t)|_\H
\end{align*}
Combining this with
\begin{align*}
	g(\nabla_X\partial_t,Y)|_\H
		&= \frac12\L_tg(X,Y)|_\H + \frac{1}{2} (g(\nabla_X\partial_t,Y) -g(\nabla_Y\partial_t,X))|_\H \\
		&= \frac12\L_tg(X,Y)|_\H - \frac{1}{2}g(\partial_t,[X,Y])
\end{align*}
and Proposition \ref{prop: first derivative} yields the desired result.
\end{proof}

\subsection{The higher derivatives}

The proof of Theorem \ref{thm: uniqueness in gauge} is an iterative construction of the $Q_m$'s.
The proof is constructive, meaning that it in principle is possible to compute the $Q_m$'s explicitly, just like we did for $Q_1$ in the previous subsection.
Remark \ref{rmk: first derivative} implies that there is a unique $Q_1$ such that
\[
	\L_tg(X, Y)|_\H = Q_1(\s, V)(X, Y)
\]
for all $X, Y \in T\H$.
Let us therefore make the following induction assumption:
\begin{ind_assumption} \label{ass: induction}
Fix an $m \in \N$.
We assume that there are unique (non-linear) diffeomorphism invariant differential operators $Q_1, \hdots, Q_m$ on $\H$ such that
\[
	\L_t^kg(X, Y) = Q_k(\sigma, V)(X, Y)
\]
for all $X, Y \in T\H$ and all $k = 1, \hdots, m$.
\end{ind_assumption}

\begin{remark} \label{rmk: m = 1 induction}
Indeed, by Remark \ref{rmk: first derivative} we have proven that Induction Assumption \ref{ass: induction} is satisfied for
\[
	m = 1.
\]
\end{remark}

Given Induction Assumption \ref{ass: induction}, the goal of this section is to show the existence of a unique $Q_{m+1}$ such that
\[
	\L_t^{m+1}g(X, Y)|_\H = Q_{m+1}(\sigma, V)(X, Y)
\]
for all $X, Y \in T\H$, which by induction would prove Theorem \ref{thm: uniqueness in gauge}.
Recall that we already know that
\[
	\L_t^mg(\d_t, \cdot) = 0,
\]
for all $m \in \N$.

\subsubsection{Notation and first identities}
It will be convenient to use the notation 
\[
	\Theta(X) = \n_X \d_t, \quad A(X, Y) := g(\Theta(X), Y).
\] 
We also define the square of a $(0,2)$-tensor $T$ as
\[
	T^2(X, Y) := \sum_{\a, \b = 0}^n g^{\alpha\beta} T(X,e_{\alpha})T(e_{\beta},Y).
\]
\begin{lemma} \label{eq: first higher derivative identities}
We have the identities
	\begin{align*}
		\nabla_tA(X,Y)
			&= -A^2(X,Y) + \Riem (\partial_t,X,\partial_t,Y), \\
		\nabla_t\L_tg(X,Y)
			&= - A^2 (X,Y) -A^2 (Y, X) +2\Riem (\partial_t,X,\partial_t,Y), \\
		\nabla_tA(X,Y)
			&=\frac 12 \nabla_t\L_tg(X,Y) - \frac12\left( A^2(X, Y) - A^2(Y, X) \right).
	\end{align*}
\end{lemma}
\begin{proof}
Since $\nabla_t\partial_t=0$, we have
	\begin{align*}
		\nabla_tA(X,Y)
			&=g(\nabla^2_{\partial_t,X}\partial_t,Y) \\
			&=g(\nabla^2_{X,\partial_t}\partial_t,Y) + \Riem(\partial_t,X,\partial_t,Y)\\
			&=-g(\nabla_{\nabla_X\partial_t}\partial_t,Y)+ \Riem (\partial_t,X,\partial_t,Y)\\
			&= - g(\Theta(\Theta(X)),Y) + \Riem(\partial_t,X,\partial_t,Y)\\
			&= - A^2(X,Y) + \Riem (\partial_t,X,\partial_t,Y).
	\end{align*}
The second and the third identities follow from the identity
\[
	\L_tg(X, Y) = A(X, Y) + A(Y, X).
\]
\end{proof}

\subsubsection{The curvature components}

The strategy in the proof of Theorem \ref{thm: uniqueness in gauge} is to show that $\L_t^{m+1}g|_\H$ is given uniquely by lower orders $\L_t^kg|_\H$ and for $k = 0, \hdots, m$, using that derivatives of the Ricci curvature vanish at the horizon.
The following lemma will be crucial for that purpose:

\begin{lemma} \label{le: curvature components}
Let $X, Y, Z, W \in C^\infty(T\H)$ and fix a $k \in \N$.
Then there is a unique way to express
\begin{enumerate}[(a)]
	\item $\n_t^j A|_\H$, for all $j = 0, \hdots, k-1$, \label{claim: A}
	\item $\left( \n_t^j\L_tg - \L_t^{j+1}g \right) |_\H$, for all $j = 0, \hdots, k$, \label{claim: Lie}
	\item $\n_t^j \Riem(\d_t, X, \d_t, Y)|_\H$, for all $j=0, \hdots, k-2$, \label{claim: first curvature}
	\item $\n_t^j \Riem(X, Y, \d_t, Z)|_\H$, for all $j=0, \hdots, k-1$, \label{claim: second curvature}
	\item $\n_t^j \Riem(X, Y, Z, W)|_\H$, for all $j=0, \hdots, k$, \label{claim: third curvature}
\end{enumerate}
in terms of $V$ and
\begin{equation} \label{eq: Lie derivatives curvature}
	g|_\H, \hdots, \L_t^kg|_\H.
\end{equation}
\end{lemma}

In the proof of this lemma and in further computations, we will use the following schematic notation:
\begin{notation}
Given two tensors $B_1$ and $B_2$, the notation
\[
	B_1 * B_2
\]
denotes a tensor, which is given by linear combinations of contractions with respect to the metric $g$ (not derivatives of $g$).
In particular, we may write
\[
	\n^k (B_1 * B_2) = \sum_{i + j = k} \n^i B_1 * \n^j B_2.
\]
\end{notation}

The proof of Lemma \ref{le: curvature components} will use the following simple observation:
\begin{lemma} \label{le: commutator n t m}
Let $S$ be a smooth tensor field on $M$. 
For any $k \in \N$, the commutator
\[
	[\nabla_{t, \hdots, t}^k, \nabla]S|_\H
\]
is determined uniquely in terms of $g|_\H$ and
\[
	\n^j_t \Riem|_\H, \quad \n^j A|_\H,
\]
for $j = 0, \hdots, k-1$ and 
\[
	\n^j S|_\H
\]
for $j = 0, \hdots, k$.
\end{lemma}
\begin{proof}[Proof of Lemma \ref{le: commutator n t m}]
Note that for any tensor field $S$, we have
\begin{align*}
	[\nabla_t,\nabla]S(X)
		&= \nabla^2_{t,X}S-\nabla_X(\nabla_tS)\\
         &= \Riem(\d_t,X)S - \n_{\Theta(X)}S.
\end{align*}
Using this and $\n_t \d_t = 0$, we may schematically compute the higher commutators as
\begin{align*}
	[\nabla_{t, \hdots, t}^k, \nabla]S
		&= [(\nabla_t)^k, \nabla]S \\
		&= \sum_{i + j = k-1} (\n_t)^i[\n_t, \n](\n_t)^j S \\
		&= \sum_{i + j = k-1} (\n_t)^i \Riem(\d_t, \cdot) \n_{t,\hdots, t}^j S - (\n_t)^i \n_{\Theta(\cdot)}\n_{t, \hdots, t}^j S \\
         &= \sum_{i + j = k-1} \sum_{a + b = i}(\n^a \Riem) * (\n^{b+j} S) - (\n^a \Theta) * (\n^{b+j+1} S).
\end{align*}
This completes the proof, since $\n^kA = g(\n^k\Theta(\cdot), \cdot)$.
\end{proof}

\begin{proof}[Proof of Lemma \ref{le: curvature components}]
We start by proving the statement for $k = 1$.
By Corollary \ref{cor: the full t derivative}, we know that
\[
	A|_\H(X, Y),
\]
for $X, Y \perp V$ is given in terms of $V$, $g|_\H$ and $\L_tg|_\H$.
The other components of $A|_\H$ are computed using that $\n_t \d_t = 0$, by construction, and that $\n_V \d_t = [V, \d_t]|_\H + \d_t W|_\H = - \kappa \d_t|_\H$, by \eqref{eq: nabla t W}.
This proves claim \eqref{claim: A} for $k = 1$.
Claim \eqref{claim: Lie} then follows for $k = 1$ immediately by noting that
\begin{equation} \label{eq: covariant vs Lie}
	\n_t S = \L_tS + A * S,
\end{equation}
for any covariant tensor $S$, from which we get
\[
	\n_t\L_tg|_\H = \L_t^2 g|_\H + A * \L_tg|_\H.
\]
Further, we compute
\begin{align*}
	\Riem(X, Y, \d_t, Z)|_\H
		&= X g(\n_Y \d_t, Z)|_\H - g(\n_Y \d_t, \n_X Z)|_\H - Yg(\n_X \d_t, Z)|_\H \\*
		&\qquad + g(\n_X \d_t, \n_Y Z)|_\H - g(\n_{[X, Y]}\d_t. Z)|_\H \\
		&= \n_X A(Y, Z)|_\H - \n_Y A(X, Z)|_\H,
\end{align*}
for all $X, Y, Z \in T\H$.
Clearly, $\n$ can be expressed in terms of $g|_\H$ and $\L_tg|_\H$ and since $X, Y$ are tangent to $\H$, we conclude claim \eqref{claim: second curvature} for $k = 1$.
The curvature component $\Riem(X, Y, Z, W)|_\H$ can locally be given in terms of Christoffel symbols and derivatives thereof tangent to $\H$, which are all given by $g|_\H$ and $\L_tg|_\H$, proving claim \eqref{claim: third curvature} for $j = 0$.

In order to complete the case $k = 1$, we still need to prove \eqref{claim: third curvature} for $j = 1$.
Using the second Bianchi identity, we compute the derivatives of the third curvature component:
\begin{align}
	\n_t^j
		&\Riem(X, Y, Z, W)|_\H \nonumber \\
		&= \n_t^{j-1} (\n\Riem)(\d_t, X, Y, Z, W)|_\H \nonumber \\
		&= - \n_t^{j-1} (\n\Riem)(X, Y, \d_t, Z, W)|_\H - \n_t^{j-1} (\n\Riem)(Y, \d_t, X, Z, W)|_\H \nonumber \\
		&= - [\n_t^{j-1}, \n] \Riem(X, Y, \d_t, Z, W)|_\H - \n_X \n_t^{j-1}\Riem (Y, \d_t, Z, W)|_\H \nonumber \\*
		&\qquad - [\n_t^{j-1}, \n] \Riem(Y, \d_t, X, Z, W)|_\H - \n_Y \n_t^{j-1} \Riem (\d_t,X, Z, W)|_\H, \label{eq: third curvature}
\end{align}
for $X, Y, Z, W \in T\H$.
Using this with $j = 1$ and that we have proven claim \eqref{claim: second curvature} for $k = 1$, we conclude claim \eqref{claim: third curvature} for $j = 1$.
We have thus proven Lemma \ref{le: curvature components} with $k = 1$, which is initial step in our induction on $k$.

Assume therefore that the assertion is proven for an $k \in \N$, we want to then prove it for $k+1$.
Lemma \ref{eq: first higher derivative identities} implies that
\begin{align*}
	\n_t^k A|_\H 
		&= \frac12 \n_t^k\L_tg|_\H + \sum_{a + b = k-1} \n_t^a A *  \n_t^b A|_\H \\
		&= \frac12 \left(\n_t^k\L_tg - \L_t^{k+1}g \right)|_\H + \frac12\L_t^{k+1}g|_\H + \sum_{a + b = k-1} \n_t^a A *  \n_t^b A|_\H,
\end{align*}
which by the induction assumption is given by
\[
	g|_\H, \hdots, \L_t^{k+1}g|_\H.
\]
proving claim \eqref{claim: A} for $k+1$.
By applying \eqref{eq: covariant vs Lie}, we note that
\begin{align*}
	\left(\n_t^{k+1} \L_tg - \L_t^{k+2}g\right)|_\H
		&= \sum_{a + b = k} \n_t^a A * \n_t^b\L_tg|_\H.
\end{align*}
By claim \eqref{claim: A} for $k + 1$, this proves claim \eqref{claim: Lie} for $k + 1$.
By Lemma \ref{eq: first higher derivative identities}, we deduce that
\[
	\n_t^{k-1} \Riem(\d_t, X, \d_t, Y)|_\H
		= \n_t^kA(X, Y)|_\H + \sum_{j = 0}^{k-1} \n_t^j A * \n_t^{k-1-j} A(X, Y)|_\H,
\]
proving claim \eqref{claim: first curvature} for $k + 1$.
To prove claim \eqref{claim: second curvature} for $k+1$, we note that the second Bianchi identity implies similar to \eqref{eq: third curvature} that
\begin{align*}
	\n_t^k
		&\Riem(X, Y, \d_t, Z)|_\H \\*
		&= - [\n_t^{k-1}, \n] \Riem(X, Y, \d_t, \d_t, Z)|_\H - \n_X \n_t^{k-1}\Riem (Y, \d_t, \d_t, Z)|_\H \\*
		&\qquad - [\n_t^{k-1}, \n] \Riem(Y, \d_t, X, \d_t, Z)|_\H - \n_Y \n_t^{k-1} \Riem (\d_t,X, \d_t, Z)|_\H,
\end{align*}
for any $X, Y, Z \in T\H$.
By Lemma \ref{le: commutator n t m} and claim \eqref{claim: first curvature} for $k+1$, the right hand side is expressed in terms of $V$ and
\[
	g|_\H, \hdots, \L_t^{k+1}g|_\H,
\]
proving claim \eqref{claim: second curvature} for $k+1$.
Finally, in order to prove claim \eqref{claim: third curvature}, we consider equation \eqref{eq: third curvature} with $j = k+1$ and get
\begin{align*}
	\n_t^{k+1}
		&\Riem(X, Y, Z, W)|_\H \\
		&= - [\n_t^k, \n] \Riem(X, Y, \d_t, Z, W)|_\H - \n_X \n_t^k\Riem (Y, \d_t, Z, W)|_\H \\*
		&\qquad - [\n_t^k, \n] \Riem(Y, \d_t, X, Z, W)|_\H - \n_Y \n_t^k \Riem (\d_t,X, Z, W)|_\H.
\end{align*}
Again, by Lemma \ref{le: commutator n t m} and claim \eqref{claim: second curvature} for $k+1$, the right hand side is expressed in terms of $V$ and
\[
	g|_\H, \hdots, \L_t^{k+1}g|_\H,
\]
proving claim \eqref{claim: third curvature} for $k+1$.

We have thus proven the assertion in Lemma \ref{le: curvature components} for $k+1$, which completes the induction argument.
\end{proof}

\subsubsection{Expressions for the $V$-components}
With the preparations in the previous subsection, we are now able to study the expression
\[
	\L_t^{m+1}g(\cdot, V)|_\H.
\]

\begin{prop} \label{prop: the V components}
Let Assumption \ref{ass: main} and Induction Assumption \ref{ass: induction} be satisfied for an $m \in \N$.
Then, for any $X \in C^\infty(T\H)$, the expression
\[
	\L_t^{m+1}g(X,V)|_\H
\]
is given uniquely in terms of
\begin{equation} \label{eq: Lie derivative XV}
	g|_\H, \hdots, \L_t^mg|_\H.
\end{equation}
\end{prop}
\begin{proof}
By Lemma \ref{eq: first higher derivative identities}, we have 
\[
	\nabla_t^m \L_tg(X,V)|_\H
		= \sum_{i + j = m-1} \n_t^i A * \n_t^j A(X,V)|_\H +2\n_t^{m-1}\Riem(\partial_t,X,\partial_t,V)|_\H
\]
for all $X \in T\H$.
The first term on the right hand side is dealt with in Lemma \ref{le: curvature components}.
The second term is computed as follows:
\begin{align*}
	0
		&= \n_t^{m-1} \Ric(X, \d_t)|_\H \\
		&= \n_t^{m-1} \tr_g\left(\Riem(\cdot, X, \d_t, \cdot)\right)|_\H \\
		&= \tr_g\left((\n_t^{m-1} \Riem)(\cdot, X, \d_t, \cdot)\right)|_\H \\
		&= \n_t^{m-1}\Riem(\d_t, X, \d_t, V)|_\H + \sum_{i,j = 2}^n g^{ij}\n_t^{m-1} \Riem(e_i, X, \d_t, e_j)|_\H.
\end{align*}
By Lemma \ref{le: curvature components}, the sum on the right hand side is given by \eqref{eq: Lie derivative XV}.
The proof is completed by applying Lemma \ref{le: curvature components}, claim \eqref{claim: Lie}.
\end{proof}

\subsubsection{Computation of the commutator}

In this section, we will compute the following commutator, which is essential for the remaining part of the proof of Theorem \ref{thm: uniqueness in gauge}:
\[
	[\n_t^{m-1}, \Box]\L_tg.
\]
\begin{prop} \label{prop: commutator}
Let Assumption \ref{ass: main} and Induction Assumption \ref{ass: induction} be satisfied for an $m \in \N$.
The expression
\[
	[\n_t^{m-1}, \Box]\L_tg|_\H + 2 (m-1) \kappa \L_t^{m+1}g|_\H
\]
is given uniquely in terms of 
\begin{equation}
	g|_\H, \hdots, \L_t^mg|_\H.  \label{eq: Lie derivatives commutator}
\end{equation}
\end{prop}
We begin with the following lemma:
\begin{lemma} \label{le: preliminary commutators}
For any tensor field $S$, we schematically have
\begin{align*}
	[\nabla_t,\nabla]S(X)
         &= \Riem(\d_t,X)S - \n_{\Theta(X)}S, \\
     [\nabla_t, \n^2]S
     	&= \n \Riem * S + \Theta * \Riem * S + \Riem * \n S + \n \Theta * \n S + \Theta * \n^2 S, \\
     [\n_t, \Box]S
     	&= g(\L_tg, \n^2)S + \n \Ric(\d_t, \cdot) * S + \Riem * \n S + \n \Theta * \n S.
\end{align*}
\end{lemma}
Here, we have used the notation
\[
	g(T,\nabla^2)S := \sum_{\a,\b,\gamma,\de = 0}^ng^{\alpha\beta}g^{\gamma\delta}T(e_\a, e_\b)\nabla^2_{e_\gamma, e_\delta}S.
\]
\begin{proof}
Note first that 
\begin{align*}
	[\nabla_t,\nabla]S(X)
		&= \nabla^2_{t,X}S-\nabla_X(\nabla_tS)\\
         &= \Riem(\d_t,X)S - \n_{\Theta(X)}S,
\end{align*}
which proves the first formula.
For the second formula, write
\begin{align*}
([\nabla_t,\nabla^2]S)(X,Y)=([\nabla_t,\nabla]\nabla S)(X,Y)+(\nabla[\nabla_t,\nabla]) S(X,Y).
\end{align*}
Inserting the above yields
\begin{align*}
	[\nabla_t,\nabla]\nabla S(X,Y)
		&= \Riem(\d_t,X) \n_YS - \n_{\Riem(\d_t,X)Y}S  - \n^2_{\Theta(X), Y}S,\\
	\nabla[\nabla_t,\nabla] S(X,Y)
		&= (\n_X \Riem)(\d_t, Y) S + \Riem(\Theta(X), Y) S + \Riem(\d_t, Y)\n_X S \\
		&\qquad - \n^2_{X, \Theta(Y)}S - \n_{(\n_X \Theta)(Y)}S,
\end{align*}
proving in particular the second formula.
We may now contract over $X$ and $Y$ to get	
\begin{align*}
	[\n_t, \Box]S
     	&= [\n_t, - \tr_g(\n^2)]S \\
     	&= - \tr_g([\n_t, \n^2])S \\
     	&= \tr_g\left( (\n_\cdot \Riem)(\d_t, \cdot) \right)S + \tr_g\left(\n^2_{\cdot, \Theta(\cdot)} + \n^2_{\Theta(\cdot), \cdot}\right)S \\
     	&\qquad + \Theta * \Riem * S + \Riem * \n S + \n \Theta * \n S.
\end{align*}
We therefore consider the endomorphism
\[
	\tr_g\left( (\n_\cdot \Riem)(\d_t, \cdot) \right)
\]
by using the second Bianchi identity s to write
\begin{align*}
	&g(\tr_g\left( (\n_\cdot \Riem)(\d_t, \cdot) \right)X, Y) \\
		&\quad = \sum_{\a, \b = 0}^n g^{\a \b} g( \n_{e_\a}\Riem(\d_t, e_\b) X, Y) \\
		&\quad = \sum_{\a, \b = 0}^n g^{\a \b} \n_{e_\a}\Riem(X, Y, \d_t, e_\b) \\
		&\quad = - \sum_{\a, \b = 0}^n g^{\a \b} \n_X \Riem(Y, e_\a, \d_t, e_\b) - \sum_{\a, \b = 0}^n g^{\a \b} \n_Y \Riem(e_\a, X,\d_t, e_\b) \\
		&\quad = \n_X \Ric(\d_t, Y) - \n_Y \Ric(\d_t, X).
\end{align*}
Finally, we use the formula
\[
	\L_tg(X, Y) = g(\Theta(X),Y)+g(X,\Theta(Y))
\]
to see that
\begin{align*}
	\tr_g\left(\n^2_{\cdot, \Theta(\cdot)} + \n^2_{\Theta(\cdot), \cdot}\right)S = g(\L_tg, \n^2)S.
\end{align*}
This completes the proof.
\end{proof}

\begin{proof}[Proof of Proposition \ref{prop: commutator}]
We first note that
\begin{align*}
	[\n_t^{m-1}, \Box] \L_tg|_\H
		&= \sum_{i + j = m-2} \n_t^i [\n_t, \Box] \n_t^j \L_tg|_\H.
\end{align*}
By applying Lemma \ref{le: curvature components}, Lemma \ref{le: preliminary commutators} and Assumption \ref{ass: main}, we note that the only term in this sum which is not determined by \eqref{eq: Lie derivatives commutator} is
\[
	\sum_{i + j = m-2} \n_t^i g(\L_tg, \n^2) \n_t^j \L_tg|_\H 
		= \sum_{i + j = m-2} \sum_{a + b = i} c(a, b) g(\n_t^a\L_tg, \n_t^b \n^2) \n_t^j \L_tg|_\H,
\]
for some combinatorial numbers $c(a, b)$.
The only term in this sum which is not determined by \eqref{eq: Lie derivatives commutator} is
\begin{align*}
	\sum_{i + j = m-2} g(\L_tg, \n_t^i \n^2) \n_t^j \L_tg|_\H
		&= \sum_{i + j = m-2} g(\L_tg, [\n_t^i, \n^2]) \n_t^j \L_tg|_\H \\*
		&\qquad + \sum_{i + j = m-2}g(\L_tg, \n^2) \n_t^i \n_t^j \L_tg|_\H \\
		&=  \sum_{i + j = m-2} \sum_{a + l = i - 1} g(\L_tg, \n_t^a [\n_t, \n^2]) \n_t^{l + j} \L_tg|_\H \\
		&\qquad + (m-1) g(\L_tg, \n^2) \n_t^{m-2} \L_tg|_\H.
\end{align*}
Again, Lemma \ref{le: preliminary commutators} implies that the only term which is not determined by \eqref{eq: Lie derivatives commutator} is
\[
	(m-1) g(\L_tg, \n^2) \n_t^{m-2} \L_tg|_\H,
		= (m-1) g^{\a \gamma} g^{\b \de} \L_tg_{\a \b} \n_{e_\de, e_\gamma}^2 \n_t^{m-2} \L_tg|_\H.
\]
By Proposition \ref{prop: first derivative} and Lemma \ref{le: preliminary commutators}, the only term in this expression which is not given by \eqref{eq: Lie derivatives commutator}, is the term
\[
	- 2 (m-1) \kappa \n_t^m\L_tg|_\H.
\]
Applying Lemma \ref{le: curvature components}, claim \eqref{claim: Lie}, completes the proof.
\end{proof}

\subsubsection{Expressions for the $V^\perp$-components}

We now turn to the computation of the final components
\begin{equation} \label{eq: V perp components}
	\L^{m+1}_t g(X, Y)|_\H,
\end{equation}
where $X, Y \in V^\perp$.
This will use the linearization of the Ricci curvature $\L_t\Ric|_\H$, found for example in \cite{Besse1987}*{Thm.\ 1.174}: 
\begin{equation} \label{eq: linearized Ricci curvature}
	2 \L_t \Ric
		= \Box_L \L_t g + \L_{\div(\L_tg - \frac12\tr_g(\L_tg)g)^{\sharp}}g,
\end{equation}
where $\Box_L$ is the d'Alembert-Lichnerowicz operator, defined as
\[
	\Box_Lh := \Box h -2 \mathring \Riem h,
\]
where 
\[
	\mathring \Riem h(X, Y) := \tr_g\left(h(\Riem(\cdot, X)Y, \cdot)\right).
\]

\begin{prop} \label{prop: Lie components V perp}
Let Assumption \ref{ass: main} and Induction Assumption \ref{ass: induction} be satisfied for an $m \in \N$.
For $X,Y \in C^{\infty}(V^\perp)$, the expression
\begin{align*}
	\L_t^{m+1}g(X, Y)|_\H
\end{align*}
is given uniquely in terms of 
\begin{equation}
	g|_\H, \hdots, \L_t^mg|_\H \quad \text{and} \quad \L_t^{m+1}g(\cdot, V)|_\H.  \label{eq: Lie derivatives lin Ricci}
\end{equation}
\end{prop}

The idea in the proof is to differentiate \eqref{eq: linearized Ricci curvature}:
\begin{align}
	2 \n_t^{m-1}\L_t\Ric|_\H
		&= \n_t^{m-1} \left( \Box \L_tg -2 \mathring \Riem \L_tg + \L_{\div((\L_tg - \frac12\tr_g(\L_tg)g)^{\sharp}}g\right)|_\H \nonumber \\
		&= [\n_t^{m-1}, \Box] \L_tg|_\H + \Box \n_t^{m-1} \L_tg|_\H -2 \n_t^{m-1} \mathring \Riem \L_tg|_\H \nonumber \\
		&\qquad + \n_t^{m-1} \L_{\div((\L_tg - \frac12\tr_g(\L_tg)g)^{\sharp}}g|_\H. \label{eq: der lin Ricci}
\end{align}
By Proposition \ref{prop: commutator}, we know that 
\[
	[\n_t^{m-1}, \Box]\L_tg|_\H + 2 (m-1) \kappa \L_t^{m+1}g|_\H
\]
is given by \eqref{eq: Lie derivatives lin Ricci}.
The remaining three terms in \eqref{eq: der lin Ricci} are treated separately in Lemma \ref{le: wave operator on derivative}, Lemma \ref{le: derivative curvature action} and Lemma \ref{le: time derivative gauge term} below.

\begin{lemma}\label{le: wave operator on derivative}
For $X,Y \in C^{\infty}(V^\perp)$, the expression
\[
	\Box\nabla_t^{m-1}\L_tg(X,Y)|_\H + 2\kappa\L^{m+1}_tg(X,Y)|_\H,
\]
is uniquely determined by \eqref{eq: Lie derivatives lin Ricci}.
\end{lemma}
\begin{proof}
We compute that
\begin{align*}
	\Box \n_t^{m-1} \L_tg|_\H
		&= - \sum_{\a,\b = 2}^n g^{\a\b}\n^2_{e_\a, e_\b} \n_t^{m-1} \L_tg|_\H \\
		&= - \n^2_{W, t} \n^{m-1}_t \L_tg|_\H - \n^2_{t, W}\n_t^{m-1} \L_tg|_\H \\
		& \qquad - \sum_{i,j = 2}^n g^{ij}\n^2_{e_i, e_j} \n^{m-1}_t \L_tg|_\H \\
		&= - 2 \n_W \n^m_t \L_tg|_\H + 2 \n_{\n_W\d_t}\n_t^{m-1} \L_tg|_\H \\
		&\qquad - \Riem(\d_t, W)\n_t^{m-1} \L_tg|_\H - \sum_{i,j = 2}^n g^{ij}\n_{e_i} \n_{e_j} \n^{m-1}_t \L_tg|_\H \\
		&\qquad + \sum_{i,j = 2}^n g^{ij}\n_{\n_{e_i}e_j} \n^{m-1}_t \L_tg|_\H.
\end{align*}
By Lemma \ref{le: curvature components} \eqref{claim: Lie} and since $\n_{e_i}e_j|_\H \in T\H$, the fourth and fifth term are uniquely given by \eqref{eq: Lie derivatives lin Ricci}.
We compute the second term using
\[
	\n_W\d_t|_\H = - \kappa \d_t|_\H,
\]
by Proposition \ref{prop: first derivative}, and get
\begin{align*}
	2 \n_{\n_W\d_t}\n_t^{m-1} \L_tg|_\H 
		&= - 2 \kappa \n_t^m \L_tg|_\H.
\end{align*}
Applying Lemma \ref{le: curvature components} \eqref{claim: Lie}, we conclude that the second term uniquely given by \eqref{eq: Lie derivatives lin Ricci}.
The first term is computed using that $W$ is a Killing vector field with $[\d_t, W] = 0$.
For $X, Y \in V^\perp$, we have
\begin{align*}
	- 2 \n_W 
		&\n^{m+1}_t \L_tg(X, Y)|_\H \\
		&= - 2 \L_W\n^{m+1}_t \L_tg(X, Y)|_\H + 2 \n^{m+1}_t \L_tg(\n_XW, Y)|_\H \\
		&\qquad + 2 \n^{m+1}_t \L_tg(X, \n_YW)|_\H \\
		&= - 2 \n^{m+1}_t \L_t\L_Wg(X, Y)|_\H \\
		&= 0.
\end{align*}
We finally consider the last term.
Note first that for $m > 1$, Lemma \ref{le: curvature components} implies that $\Riem|_\H$ is determined by \eqref{eq: Lie derivatives lin Ricci}.
By Lemma \ref{le: curvature components} \eqref{claim: Lie}, we conclude that in this case, the last term is determined by \eqref{eq: Lie derivatives lin Ricci}.
In case $m = 1$, Proposition \ref{prop: first derivative} implies that for any $X, Y \in V^\perp$, we have
\begin{align*}
	- \Riem(\d_t, W)\L_tg(X, Y)|_\H
		&= \L_tg(\Riem(\d_t, W) X, Y)|_\H + \L_tg(X, \Riem(\d_t, W) Y)|_\H \\
		&= \sum_{i,j=2}^n\Riem(\d_t, W, X, e_i)g^{ij}\L_tg(e_j, Y)|_\H \\*
		&\qquad + \sum_{i,j = 2}^n \Riem(\d_t, W, X, e_i)g^{ij} \L_tg(X, e_j)|_\H.
\end{align*}
These curvature components are determined by \eqref{eq: Lie derivatives lin Ricci} and by Lemma \ref{le: curvature components} \eqref{claim: second curvature}.
Taken together, this completes the proof.
\end{proof}

\begin{lemma}\label{le: derivative curvature action}
For all $X, Y \in V^\perp$, the expression
\[
	\nabla_t^{m-1}(\mathring \Riem \L_tg)(X,Y)|_\H
		- \kappa\L_t^{m+1}g(X,Y)|_\H
\]
is given uniquely by \eqref{eq: Lie derivatives lin Ricci}.
\end{lemma}
\begin{proof}
We compute
\begin{align*}
	&\nabla_t^{m-1}\mathring \Riem \L_tg(X,Y)|_\H \\
		&= \sum_{\a,\b,\gamma,\de = 0}^n \sum_{k=0}^{m-1}\binom{m-1}{k}g^{\alpha\beta}g^{\gamma\delta}\nabla_t^k \Riem (e_{\alpha},X,Y,e_{\gamma})\nabla_{t}^{m-1-k}\L_tg(e_{\beta},e_{\delta})|_\H \\
		&=\nabla_t^{m-1} \Riem (\partial_t,X,Y,\partial_t)\L_tg(V,V)|_\H \\*
		&\qquad + \sum_{i,j = 2}^n \sum_{a,b = 2}^n g^{ij}g^{ab}\nabla_t^{m-1} \Riem(e_i,X,Y,e_a)\L_tg(e_j,e_b)|_\H \\*
		&\qquad + \sum_{k=0}^{m-2}\nabla_t^k \Riem * \nabla_t^{m-1-k}\L_tg(X, Y)|_\H.
\end{align*}
By Lemma \ref{le: curvature components}, the second and the third terms are given by \eqref{eq: Lie derivatives lin Ricci}.
The first term is computed using that $\L_tg(V,V)|_\H = -2\kappa$, by Proposition \ref{prop: first derivative}, and Lemma \ref{eq: first higher derivative identities}:
\begin{align*}
	\nabla_t^{m-1} \Riem (\partial_t,X,Y,\partial_t)\L_tg(V,V)|_\H
		&= \kappa \n_t^m\L_tg(X, Y)|_\H \\*
		&\qquad + \sum_{k = 0}^{m-1} \n_t^k A * \n_t^{m - 1 - k} A(X, Y)|_\H.
\end{align*}
By Lemma \ref{le: curvature components} \eqref{claim: A}, the sum is given uniquely by \eqref{eq: Lie derivatives lin Ricci}.
The proof is hence finished by applying Lemma \ref{le: curvature components} \eqref{claim: Lie}.
\end{proof}

\begin{lemma}\label{le: time derivative gauge term}
For any $X, Y \in T\H$, we have
	\[
		\nabla^{m-1}_t \L_{\div(\L_tg - \frac12\tr_g(\L_tg)g)^{\sharp}}g(X,Y)|_\H 
	\]
is given uniquely by \eqref{eq: Lie derivatives lin Ricci}.
\end{lemma}
\begin{proof}
We have
\begin{align}
	\nabla^{m-1}_t 
		&\L_{\div(\L_tg - \frac12\tr_g(\L_tg)g)^{\sharp}}g(X,Y)|_\H \nonumber \\
		&= g(\nabla^m_{t, \hdots, t, X} \div(\L_tg - \frac12\tr_g(\L_tg)g)^{\sharp},Y)|_\H \nonumber \\ 
		&\qquad + g(X, \nabla^m_{t, \hdots, t, Y} \div(\L_tg - \frac12\tr_g(\L_tg)g)^{\sharp})|_\H \nonumber \\
		&= \sum_{\a, \b = 0}^n g^{\alpha\beta}(\nabla^{m+1}_{t,\ldots,t,X,e_{\alpha}}(\L_tg)(e_{\beta},Y)|_\H + \sum_{\a, \b = 0}^n g^{\alpha\beta}\nabla^{m+1}_{t,\ldots, t,Y,e_{\alpha}}(\L_tg)(e_{\beta},X))|_\H \nonumber \\*
		&\qquad - \sum_{\a, \b = 0}^n g^{\alpha\beta}\nabla^{m+1}_{t,\ldots,t,X,Y}(\L_tg)(e_{\alpha},e_{\beta})|_\H \nonumber \\
		&= \nabla^{m+1}_{t,\ldots,t,X,t}(\L_tg)(V,Y)|_\H + \nabla^{m+1}_{t,\ldots,t,Y,t}(\L_tg)(V,X)|_\H \nonumber \\*
		&\qquad + \nabla^{m+1}_{t,\ldots,t,X,V}(\L_tg)(\d_t,Y)|_\H + \nabla^{m+1}_{t,\ldots,t,Y,V}(\L_tg)(\d_t,X)|_\H \nonumber \\*
	&\qquad + \sum_{i, j = 2}^n g^{ij} \nabla^{m+1}_{t,\ldots,t,X,e_{i}}(\L_tg)(e_{j},Y)|_\H + \sum_{i, j = 2}^n g^{ij} \nabla^{m+1}_{t,\ldots,t,Y,e_{i}}(\L_tg)(e_{j},X)|_\H \nonumber \\*
	&\qquad - \sum_{\a, \b = 0}^n g^{\a\b}\nabla^{m+1}_{t,\ldots,t,X,Y}(\L_tg)(e_\a,e_\b)|_\H. \label{eq: derivative of Lie part}
\end{align}
We begin by noting that 
\begin{equation} \label{eq: VY component}
	\nabla^{m+1}_{t,\ldots,t,X,t}(\L_tg)(V,Y)|_\H
		= [\n_t^{m-1}, \n]\n_t\L_tg(X, V, Y)|_\H + \n_X \n_t^m\L_tg(V, Y)|_\H.
\end{equation}
Since $X \in T\H$, the second term in \eqref{eq: VY component} is given by \eqref{eq: Lie derivatives lin Ricci}.
We compute the first term in \eqref{eq: VY component} as follows:
\begin{align*}
	[\n_t^{m-1}, \n]\n_t \L_tg(X, V, Y)|_\H
		&= \sum_{i + j = m-2} \n_t^i [\n_t, \n] \n_t^{j + 1} \L_tg(X, V, Y)|_\H \\
		&= \sum_{i + j = m-2} \n_t^i \left(\Riem(\d_t, \cdot) - \n_{\Theta(X)}\right) \n_t^{j + 1} \L_tg(X, V, Y)|_\H \\
		&= \sum_{i + j = m - 2} \n_t^i \Riem * \n_t^{j + 1}\L_tg(X, Y)|_\H \\
		&\qquad + \sum_{i + j = m-2} \sum_{a = 1}^i \n_t^a \Theta * \n_t^{i - a} \n \n_t^{j + 1}\L_tg(X, Y)|_\H \\
		&\qquad - \sum_{i + j = m - 2} \n_{\underbrace{t, \hdots, t}_{i \text{ times}}, \Theta(X), \underbrace{t, \hdots, t}_{j + 1 \text{ times}}}^m\L_tg(V, Y)|_\H.
\end{align*}
Applying Lemma \ref{le: curvature components}, note that only the last term is not obviously given by \eqref{eq: Lie derivatives lin Ricci}, we compute it separately:
\begin{align*}
	&\sum_{i + j = m - 2} \n_{\underbrace{t, \hdots, t}_{i \text{ times}}, \Theta(X), \underbrace{t, \hdots, t}_{j + 1 \text{ times}}}^m\L_tg(V, Y)|_\H \\
		& \quad = \sum_{i + j = m - 2} [\n^i_t, \n] \n^{j+1}_t \L_tg(\Theta(X), V, Y)|_\H + (m - 1) \n^m_{\Theta(X), t, \hdots, t} \L_tg(V, Y)|_\H.
\end{align*}
Arguing as above for $[\n_t^i, \n]$, we see that the sum is given by \eqref{eq: Lie derivatives lin Ricci}.
Since $X \in C^\infty(V^\perp)$, we know that $\Theta(X) \in C^\infty(T\H)$, which implies that the second term in the above expression is given by \eqref{eq: Lie derivatives lin Ricci}.
Thus we have proven that the first and second terms in \eqref{eq: derivative of Lie part} are given by \eqref{eq: Lie derivatives lin Ricci}.

We turn to the remaining terms in \eqref{eq: derivative of Lie part}.
Note first that
\begin{align*}
	\n_{t, \hdots, t, \cdot, \cdot}^{m+1} \L_tg
		&= [\n_t^{m-1}, \n^2]\L_tg + \n^2 \n_t^{m-1}\L_tg \\
		&= [\n_t^{m-1}, \n] \n \L_tg + \n [\n_t^{m-1}, \n] \L_tg + \n^2 \n_t^{m-1}\L_tg \\
		&= [[\n_t^{m-1}, \n], \n] \L_tg + 2 \n [\n_t^{m-1}, \n] \L_tg + \n^2 \n_t^{m-1}\L_tg.
\end{align*}
Analogous to above, one checks that $[\n_t^{m-1}, \n] \L_tg$ and $\n_t^{m-1}\L_tg$ are given by \eqref{eq: Lie derivatives lin Ricci}.
Hence for $X, Y \in C^\infty(T\H)$, we have $\n_X Y \in T\H$ and we conclude that
\[
	2 \n_X [\n_t^{m-1}, \n] \L_tg(Y)|_\H + \n^2_{X, Y} \n_t^{m-1}\L_tg|_\H
\]
is given by \eqref{eq: Lie derivatives lin Ricci}.
We also note that
\begin{align*}
	[[\n_t^{m-1}, \n], \n]\L_tg
		&= \sum_{i+j=m-2} [\n_t^i \Riem * \n_t^j, \n] \L_tg + \sum_{i+j=m-2} [\n_t^i \Theta * \n^{j + 1}_{t, \hdots, t, \cdot} , \n]\L_tg.
\end{align*}
Applying Lemma \ref{le: curvature components} gives, by the same arguments as above, that also this is given by \eqref{eq: Lie derivatives lin Ricci}.
We conclude that for any $X, Y \in C^\infty(T\H)$
\[
	\n_{t, \hdots, t, X, Y}\L_tg|_\H
\]
is given by \eqref{eq: Lie derivatives lin Ricci}.
This completes the proof.
\end{proof}

We may now prove Proposition \ref{prop: Lie components V perp}:

\begin{proof}[Proof of Proposition \ref{prop: Lie components V perp}]
Combining equation \eqref{eq: der lin Ricci} with Proposition \ref{prop: commutator}, Lemma \ref{le: wave operator on derivative}, Lemma \ref{le: derivative curvature action} and Lemma \ref{le: time derivative gauge term}, we conclude that for all $X, Y \in V^\perp$
\[
	2 \n_t^{m-1}\L_t\Ric(X, Y)|_\H + 2 (m + 1) \kappa \L_t^{m+1}g(X, Y)|_\H
\]
is given uniquely by \eqref{eq: Lie derivatives lin Ricci}.
Since by Assumption \ref{ass: main}, we have
\[
	2 \n_t^{m-1}\L_t\Ric(X, Y)|_\H = 0
\]
and $\kappa \neq 0$, we conclude that 
\[
	\L_t^{m+1}g(X, Y)|_\H
\]
is given uniquely by \eqref{eq: Lie derivatives lin Ricci}.
This completes the proof.
\end{proof}

We may finally conclude the proof of Theorem \ref{thm: uniqueness in gauge}:

\begin{proof}[Proof of Theorem \ref{thm: uniqueness in gauge}]
As we pointed out in Remark \ref{rmk: m = 1 induction}, Proposition \ref{prop: first derivative} implies that the assertion in Induction Assumption \ref{ass: induction} for
\[
	m = 1.
\]
We thus assume that Induction Assumption \ref{ass: induction} is true for a fixed
\[
	m \in \N
\]
and aim to prove the assertion in Induction Assumption \ref{ass: induction} for
\[
	m+1.
\]
By Proposition \ref{prop: the V components}, we know that there is a unique way to express the components
\[
	\L_t^{m+1}g(X, V)|_\H
\]
in terms of $g|_\H, \hdots, \L_t^mg|_\H$.
Since by Induction Assumption \ref{ass: induction}, we know that 
\begin{align*}
	\L_t^kg(X, Y)|_\H
		&= Q_k(\s, V)(X, Y)
\end{align*}
for all $k = 0, \hdots, m$ and $X, Y \in T\H$, and since Lemma \ref{le: d t components} implies that
\[
	\L_t^kg(\d_t, \cdot) |_\H = 0
\]
for all $k \in \N$, we conclude that there is a unique $Q_{1,m+1}(\s, V)$ such that
\[
	\L_t^{m+1}g(V, X)|_\H = Q_{1, m+1}(\s, V)(V, X)
\]
for all $X \in T\H$.
Similarly, by Proposition \ref{prop: Lie components V perp}, we conclude that there is a $Q_{2,m+1}(\s, V)$ such that
\[
	\L_t^{m+1}g(X, Y)|_\H = Q_{2, m+1}(\s, V)(X, Y)
\]
for all $X, Y \in V^\perp$.
To sum up, we have shown that there is a unique $Q_{m+1}(\s, V)$ such that
\[
	\L_t^{m+1}g(X, Y)|_\H 
		= Q_{m+1}(\s, V)(X, Y)
\]
for all $X, Y \in T\H$.
Finally, note that $Q_{m+1}$ constructed this way is diffeomorphism invariant.
This completes the induction argument and finishes the proof.
\end{proof}

\subsection{Proof of geometric uniqueness} \label{subsec: uniqueness}

The goal of this section is to prove Theorem \ref{thm: main uniqueness formal}.
We are given two smooth vacuum solutions $M$ and $\hat M$ with embeddings of $\H$
\begin{center}
\begin{tikzcd}
\H \arrow[hookrightarrow]{dr}{\hat \iota} \arrow[hookrightarrow]{r}{\iota} & M \\
 & \hat M
\end{tikzcd}
\end{center}
as non-degenerate Killing horizons with horizon Killing vector fields $W$ and $\hat W$, respectively, such that $(\H, \s, V)$ is the induced data in both cases.\footnote{In the previous subsections we for simplicity wrote $\H \subset M$, but here we emphasize the embedding, since we study two solutions when proving Theorem \ref{thm: main uniqueness formal}.}
We want to construct open neighborhoods
\[
	\iota(\H) \subset \U \subset M
\]
and 
\[
	\hat \iota (\H) \subset \hat \U \subset \hat M
\]
and an diffeomorphism
\[
	\Phi : (\U, g|_\U) \to  (\hat \U, \hat g|_{\hat \U})
\]
such that
\begin{align*}
	\Phi \circ \iota 
		&= \hat \iota, \\
	\Phi^* \hat W 
		&= W,
\end{align*}
and with
\begin{align*}
	\n^k \left( \Phi^*\hat g - g\right)|_{\i(\H)}
		&= 0
\end{align*}
for all $k \in \N_0$.
The question is how $\Phi$ should be constructed away from $\iota(\H)$.
If $\iota(\H) \subset M$ was a spacelike or timelike hypersurface, a natural way to construct $\Phi$ near $\iota(\H)$ would be to map the unit speed geodesics normal to $\iota(\H)$ to the unit speed geodesics normal to $\hat \iota (\H)$.
Since $\iota(\H) \subset M$ is a \emph{lightlike} hypersurface, the normal vector field is \emph{tangent} to $\iota(\H)$, so that method does not work.
However, recall from Proposition \ref{prop: the d_t vector field} that there is a \emph{geometrically canonical} \emph{transversal} vector field $\d_t|_{\iota(\H)}$ along $\iota(\H)$ and $\d_{\hat t}|_{\hat \iota(\H)}$ along $\hat \iota(\H)$, which plays an analogous role as a unit normal does to a timelike or spacelike hypersurface.
This idea together with the iterative determination of the asymptotic expansion in the null time function, Theorem \ref{thm: uniqueness in gauge}, are the keys in our proof of Theorem \ref{thm: main uniqueness formal}.

We let $\d_t$ and $\d_{\hat t}$ denote the smooth vector fields given by Proposition \ref{prop: the d_t vector field}, defined in open neighborhoods 
\[
	\iota(\H) \subset \U \subset M \quad \text{and} \quad \hat \iota (\H) \subset \hat \U \subset \hat M,
\]
respectively.

\begin{prop}[The candidate isometry] \label{prop: candidate isometry}
Assume the same as in Theorem \ref{thm: main uniqueness formal}.
Shrinking $\U$ and $\hat \U$ if necessary, there is a unique diffeomorphism
\[
	\Phi: \U \to \hat \U
\]
such that
\begin{align}
	\Phi|_{\iota(\H)} 
		&= \hat \iota \circ \iota^{-1}|_{\iota(\H)} , \label{eq: Phi restricted varphi}\\
	\md \Phi(\d_t)
		&= \d_{\hat t}. \label{eq: geodesics}
\end{align}
Moreover, for this diffeomorphism, we have 
\begin{equation} \label{eq: Killing vector field preserved}
	\Phi^* \hat W
		= W.
\end{equation}
\end{prop}

\begin{proof}
Let $q \in \U$ be given.
By Proposition \ref{prop: the d_t vector field}, there is a unique integral curve of $\d_t$, starting at some point $p \in \iota(\H)$ and reaching $q$ after eigentime $t(q)$.
Let $\hat p := \hat \iota \circ \iota^{-1}(p) \in \hat \iota(\H)$ and consider the integral curve of $\d_{\hat t}$ in $\hat \U$.
Let $\hat q \in \hat \U$ denote point which the integral curve of $\d_{\hat t}$ reaches after time $t(q)$, i.e.\ the unique point on the integral curve starting in $\hat p$ with 
\[
	\hat t(\hat q) = t(q).
\]
In order for this to be defined, we shrink $\U$ if necessary, i.e.\ assume that $q$ is sufficiently near $\iota(\H)$. 
For each $q \in \U$, set 
\[
	\Phi(q) := \hat q.
\]
It follows that $\Phi$ is smooth and we may shrink $\hat \U$ to make sure that $\Phi$ is bijective.
The inverse is also smooth, since it is constructed the same way.
The properties \eqref{eq: Phi restricted varphi} and \eqref{eq: geodesics} follow by construction.
Finally, \eqref{eq: Killing vector field preserved} is now immediate from \eqref{eq: dt and W commuting} and the properties \eqref{eq: Phi restricted varphi} and \eqref{eq: geodesics}.
\end{proof}

We may finally prove Theorem \ref{thm: main uniqueness formal}, which says that $\Phi$ is an isometry to infinite order at the horizon and an isometry in an open neighborhood if the metrics are real analytic.

\begin{proof}[Proof of Theorem \ref{thm: main uniqueness formal}]
Let $\Phi$ be as in Proposition~\ref{prop: candidate isometry}.
We claim that 
\begin{equation} \label{eq: asymptotic uniqueness}
	\L_t^m \left( \Phi^* \hat g - g \right)|_{\i(\H)} = 0
\end{equation}
for all $m \in \N_0$.
By real analyticity, this will suffice to prove the assertion.
We begin with $m = 0$.
Note that
\begin{align*}
	\md \Phi(\d_t)
		&= \d_{\hat t}, \\
	\md \Phi(\md \iota(V))
		&= \md \hat \iota (V)
\end{align*}
and 
\begin{equation} \label{eq: m=0 check}
	\md \Phi|_{\md \iota \left( V^\perp \right)} 
		= \md \hat \iota \circ (\md \iota)^{-1}|_{\md \iota \left( V^\perp \right)} : \md \iota \left( V^\perp \right) \to \md \hat \iota \left( V^\perp \right)
\end{equation}
is an isomorphism.
What remains is therefore to check that \eqref{eq: m=0 check} is an isometry.
Given $X, Y \in V^\perp$, we compute
\begin{align*}
	\Phi^* \hat g(\md \iota(X), \md \iota(Y))
		&= \hat g(\md \Phi \circ \md \iota(X), \md \Phi \circ \md \iota(Y)) \\
		&= \hat g(\md \hat \iota (X), \md \hat \iota (Y)) \\
		&= \s(X, Y) \\
		&= g(\md \iota(X), \md \iota(Y)).
\end{align*}
This completes the proof of \eqref{eq: asymptotic uniqueness} when $m = 0$.
Let us now turn to the case $m \geq 1$.
Proposition \ref{prop: candidate isometry} implies that 
\[
	\L_t \Phi^* = \Phi^* \L_{\md \Phi(\d_t)} = \Phi^* \L_{\d_{\hat t}} = \Phi^* \L_{\hat t} 
\]
and hence
\[
	\L_t^m \Phi^* = \Phi^* \L_{\hat t}^m
\]
for all $m \in \N$.
The claim \eqref{eq: asymptotic uniqueness} can thus be rewritten as
\begin{equation} \label{eq: asymptotic uniqueness version}
	\Phi^*\L_{\hat t}^m  \hat g|_{\i(\H)} = \L_t^m g|_{\i(\H)}.
\end{equation}
We first prove that
\begin{equation} \label{eq: isometry d t}
	\Phi^*\L_{\hat t}^m  \hat g(\d_t, \cdot)|_{\i(\H)} = \L_t^m g(\d_t, \cdot)|_{\i(\H)}.
\end{equation}
The right hand side of this is vanishing by Lemma \ref{le: d t components}.
The left hand side is computed using Proposition \ref{prop: candidate isometry} as
\begin{align*}
	\Phi^*\L_{\hat t}^m  \hat g(\d_t, \cdot)|_{\i(\H)}
		&= \L_{\hat t}^m  \hat g(\md \Phi(\d_t), \md \Phi(\cdot))|_{\Phi (\i(\H))} \\
		&= \L_{\hat t}^m  \hat g(\d_{\hat t}, \md \Phi(\cdot))|_{\hat \i(\H)},
\end{align*}
which is vanishing as well by Lemma \ref{le: d t components}, proving \eqref{eq: isometry d t}.
We now prove that
\begin{equation} \label{eq: induced components}
	\iota^*\Phi^*\L_{\hat t}^m  \hat g
		= \iota^*\L_t^m g,
\end{equation}
which together with \eqref{eq: isometry d t} would prove \eqref{eq: asymptotic uniqueness version} and thus prove \eqref{eq: asymptotic uniqueness}.
By Proposition \ref{prop: candidate isometry} and Theorem \ref{thm: uniqueness in gauge}, we have
\begin{align*}
	\i^*\Phi^*\L_{\hat t}^m  \hat g
		&= (\Phi \circ \i)^*\L_{\hat t}^m  \hat g \\
		&= \hat \i^*\L_{\hat t}^m  \hat g \\
		&= Q_m(\s, V) \\
		&= \iota^* \L_t^mg,
\end{align*}
which completes the proof of \eqref{eq: induced components}.
This completes the proof of \eqref{eq: asymptotic uniqueness}.

If the spacetime metrics $g$ and $\hat g$ are real analytic, elliptic regularity theory implies that the Killing vector fields $W$ and $\hat W$ are real analytic (since $\L_Wg$ is the symmetric derivative, which is a PDE with injective principle symbol).
Hence the data $(\H, \s, V)$ are real analytic and it follows that the vector fields $\d_t$ and $\d_{\hat t}$ are real analytic and therefore the diffeomorphism $\Phi$ is real analytic with real analytic inverse.
Since
\begin{align*}
	\n^k \left( \Phi^*\hat g - g\right)|_{\i(\H)}
		&= 0
\end{align*}
for all $k \in \N_0$, real analyticity implies that 
\begin{align*}
	\Phi^*\hat g - g
		&= 0
\end{align*}
in an open neighborhood of $\iota(\H)$, which is the last assertion in Theorem \ref{thm: main uniqueness formal}.
\end{proof}

\section{Existence of the expansion} \label{sec: existence}

The goal of this section is to prove Theorem \ref{thm: main existence formal}.
The proof will be a combination of Theorem \ref{thm: main uniqueness formal} and an existence theorem by Moncrief in \cite{M1982}, which we recall here in the form we need it.
We use a slightly different notation than in \cite{M1982} to better match the notation of the present paper:
\begin{thm}[A version of Moncrief's existence theorem] \label{thm: Moncrief}
Let $\V \subset \R^2$ be an open region with coordinates $x_2, x_3$.
Given a smooth function $\mathring \phi$, a smooth one-form $\mathring \b$ and a smooth Riemannian metric $\mathring h$ on $\V$, there is a smooth spacetime metric of the form
\begin{equation} \label{eq: Moncrief gauge}
	g 
		:= e^{-2\phi} \left( - \frac{N^2 \md \tau^2}{4 \tau} + \sum_{a,b = 2}^3 h_{ab} \md x^a \md x^b\right) + \tau e^{2\phi} \left( \frac1{2\tau} \md \tau + \frac12 \md x^1 + \sum_{b = 2}^3 \b_a \md x^a \right)^2,
\end{equation}
on 
\[
	M := (-\e, \e)_\tau \times S^1_{x_1} \times \V_{x_2, x_3},
\]
such that
\[
	\n^k \Ric_g|_{\tau = 0} = 0
\]
for all $k \in \N_0$, where
\[
	N := \frac{e^{2 \mathring \phi}}{\sqrt{\det(\mathring h_{ij})}}\sqrt{\det(h_{ij})} 
\]
and where $\phi, \b$ and $h$ are a smooth function, one-form and Riemannian metric, respectively, such that
\begin{align*}
	\d_{x_1} \phi 
		&= 0, \quad \L_{\d_{x_1}} \b = 0, \quad \L_{\d_{x_1}} h = 0,	\\
	\phi|_{\tau = 0, x_1 = c}
		&= \mathring \phi, \quad \b|_{\tau = 0, x_1 = c} = \mathring \b, \quad h|_{\tau = 0, x_1 = c} = \mathring h,
\end{align*}
for any $c \in S^1$.
Moreover, the hypersurface
\[
	\{\tau = 0\} \times S^1_{x_1} \times \V_{x_2, x_3}
\]
is a smooth non-degenerate Killing horizon.
In addition, if the data $\mathring \phi$, $\mathring \b$ and $\mathring h$ are real analytic, then $g$ is real analytic and 
\[
	\Ric_g = 0
\]
in $M$.
\end{thm}
\begin{proof}
In case we replace $\V$ with $T^2$, the analytic version of the above is precisely the statement in Theorem \cite{M1982}*{Thm.\ 2}.
However, Moncrief proves the analytic version of the above statement of Theorem \ref{thm: Moncrief} as an intermediate step, see \cite{M1982}*{p.\ 90}.
The fact that one can construct the series expansion in the smooth case follows from Moncrief's derivation of the Einstein vacuum equations in his coordinates in \cite{M1982}*{p.\ 87-88}.
\end{proof}

We use Theorem \ref{thm: Moncrief} to prove a local version of Theorem \ref{thm: main existence formal}. 
This amounts to choosing Moncrief's data $\mathring \phi$, $\mathring \b$ and $\mathring h$ in terms of our data $\s$ and $V$.

\begin{lemma} \label{le: local existence}
Let $(\H, \sigma, V)$ be as in Theorem \ref{thm: main existence formal}.
For every $p \in \H$, there is an open subset
\[
	p \in \tilde \H \subseteq \H
\]
and a smooth spacetime $(M,g)$, a smooth Killing vector field $W$ on $M$ and a smooth non-degenerate Killing horizon
\[
	\iota: \tilde \H \hookrightarrow M,
\]
with respect to $W$, such that
\begin{itemize}
	\item $\n^k \Ric_g|_{\i(\tilde \H)} = 0$ for all $k \in \N_0$,
	\item $\left({\tilde \H}, \sigma|_{\tilde \H}, V|_{\tilde \H}\right)$ is the induced data as in Definition \ref{def: induced data}.
\end{itemize}
If in addition $\s$ is real analytic, then $\iota$ is real analytic and
\[
	\Ric_g = 0
\]
in an open neighborhood of $\iota(\tilde \H)$.
\end{lemma}
\begin{proof}
For a small enough $\e > 0$, let
\[
	\tilde \H \cong (-\e, \e)^3
\]
be a coordinate neighborhood around $p$, such that
\[
	\d_{x_1} = V.
\]
We write
\[
	\s|_{\tilde \H} = \s_{ij}\md x^i \md x^j,
\]
and conclude that $\s_{ij}$ are independent of $x_1$.
We may thus extend $\s$ to a metric on
\[
	S^1_{x_1} \times (-\e, \e)^2_{x_2, x_3},
\]
by the same formula, since all $\s_{ij}$ are independent of $x_1$.
We get an isometric embedding
\[
	\tilde \H \hookrightarrow S^1_{x_1} \times (-\e, \e)^2_{x_2, x_3},
\]
which respects the Killing vector field.
It therefore suffices to prove the assertion for the latter, so we can without loss of generality assume that
\[
	\tilde \H = S^1_{x_1} \times \V_{x_2, x_3},
\]
where 
\[
	\V_{x_2, x_3} := (-\e, \e)^2_{x_2, x_3}.
\]
Recall that we want
\[
	\s|_\U = \iota^* g|_{\tilde \H} + \o \otimes \o|_{\tilde \H},
\]
and recall from Section \ref{sec: sigma g omega} that $\s$ and $V$ determine $\iota^* g$ and $\o$ uniquely at every point on the horizon.
We therefore need to compute what $\iota^* g$ and $\o^2$ are in Moncrief's expression \eqref{eq: Moncrief gauge}.
Setting $\tau = 0$ in \eqref{eq: Moncrief gauge} and disregarding the components containing $\md \tau$ in $g$, we conclude that
\begin{equation} \label{eq: iota g}
	\iota^*g 
		= e^{-2 \mathring \phi} \mathring h_{ab}\md x^a \md x^b.
\end{equation}
In order to compute $\o$, we recall the defining equation and get
\begin{align*}
	\o(X)V|_{\tilde \H} 
		&= \n_XW|_{\tilde \H} \\
		&= \n_{\sum_{i =1}^3 X^i\d_{x_i}}\d_{x_1}|_{\tau = 0} \\
		&= \sum_{i =1}^3 X^i \Gamma_{i1}^1 V.
\end{align*}
We compute
\begin{align*}
	\Gamma_{i1}^1|_{\tau = 0} 
		&= \frac12 \sum_{\a = 0}^3 g^{1\a}\left( \d_i g_{\a1} + \d_1 g_{i\a} - \d_\a g_{i1} \right)|_{\tau = 0} \\
		&= \frac12 g^{10}\left( \d_i g_{01} - \d_0 g_{i1} \right)|_{\tau = 0},
\end{align*}
where $x_0 := \tau$.
Using that 
\[
	g_{10}|_{\tau = 0} = \frac{e^{2\mathring \phi}}4, \quad g_{1i}|_{\tau = 0} = 0,
\]
for $i = 1, 2, 3$, we note that
\[
	1 = \sum_{\a = 0}^3 g_{1 \a}g^{\a1}|_{\tau = 0} = \frac{e^{2 \mathring \phi}}4 g^{01}|_{\tau = 0},
\]
from which it follows that
\begin{align*}
	g^{01}|_{\tau = 0} 
		&= 4 e^{-2\mathring \phi}, \quad
	\d_i g_{01}|_{\tau =0}
		= \frac{e^{2\mathring \phi}}2\d_i \mathring \phi.
\end{align*}
It remains to compute, for $a = 2, 3$,
\begin{align*}
	- \d_0 g_{11} |_{t = 0}
		&= - \d_\tau \left( \frac{\tau e^{2\phi}}4 \right)|_{\tau = 0} 
		= - \frac{e^{2\mathring \phi}}4 \\
	- \d_0 g_{a1} |_{t = 0}
		&= - \d_\tau \left(\frac{\tau \b_a e^{2 \phi}}2\right)|_{\tau = 0} 
		= - \frac{\mathring \b_a e^{2 \mathring \phi}}2.
\end{align*}
Putting this together, we conclude that
\begin{equation} \label{eq: omega}
	\o = - \frac12 \md x^1 + (\d_2 \mathring \phi - \mathring \b_2 ) \md x^2 + (\d_3 \mathring \phi - \mathring \b_3) \md x^3.
\end{equation}
Recall that $\iota^*g$ and $\o$ are given by $\s$ and $V$ and we want to construct solutions to \eqref{eq: iota g} and \eqref{eq: omega}.
Note, however, that this system is underdetermined. 
Indeed, if we choose $\mathring \phi = 0$, there is a unique way of choosing $\mathring h_{ab}$ and $\b_a$ satisfying \eqref{eq: iota g} and \eqref{eq: omega}.
By Theorem \ref{thm: Moncrief}, we conclude that there is a spacetime metric $g$ of the form \eqref{eq: Moncrief gauge} and a horizon Killing vector field
\[
	W := \d_{x_1},
\]
such that
\begin{itemize}
	\item $\n^k \Ric_g|_{\i(\tilde \H)} = 0$ for all $k \in \N_0$,
	\item $\left({\tilde \H}, \sigma|_{\tilde \H}, V|_{\tilde \H}\right)$ is the induced data as in Definition \ref{def: induced data}.
\end{itemize}
If $\s$ is real analytic, then elliptic theory implies that $V$ is real analytic.
Therefore our choices above clearly imply that the data $\mathring \phi$, $\mathring \b$ and $\mathring h$ is real analytic.
Hence Theorem \ref{thm: Moncrief} implies that the spacetime metric $g$ is real analytic and we consequently have
\[
	\Ric_g = 0
\]
in an open neighborhood of $\iota (\tilde \H)$.
This completes the proof.
\end{proof}

\begin{remark} \label{rmk: function counting}
We proved in Theorem \ref{thm: main uniqueness formal} that the data $(\H, \s, V)$ characterizes the vacuum spacetime uniquely near the horizon.
As can be seen in the proof of Lemma \ref{le: local existence}, there are many choices of data in Theorem \ref{thm: Moncrief} that lead to isometric (i.e.\ the same) vacuum spacetime, i.e.\ we can without loss of generality set $\mathring \phi = 0$.
\end{remark}

We may finally prove Theorem \ref{thm: main existence formal} by applying Lemma \ref{le: local existence} to each open subset in a cover of the horizon and then gluing together using our Theorem \ref{thm: main uniqueness formal}:

\begin{proof}[Proof of Theorem \ref{thm: main existence formal}]
We are given data
\[
	(\H, \s, V),
\]
and we would like to define the smooth spacetime metric $g$ on the manifold
\[
	\R \times \H,
\]
in an open neighborhood of $\{0\} \times \H$.
Let $p \in \H$.
By Lemma \ref{le: local existence}, there is an open neighborhood $\tilde \H_p \subset \H$ of $p$ and a spacetime $M_p$ such that 
\[
	\iota_p: \tilde \H_p \to M_p
\]
is a non-degenerate Killing horizon with induced data (as in Definition \ref{def: induced data})
\[
	(\tilde \H_p, \s|_{\tilde \H_p}, V|_{\tilde \H_p}).
\]
By shrinking $\tilde \H_p$ and $M_p$ if necessary, Proposition \ref{prop: the d_t vector field} and Proposition \ref{prop: null time function} provide a diffeomorphism
\[
	M_p \cong (-\e_p, \e_p) \times \tilde \H_p
\]
for a small $\e_p > 0$, such that $\d_t$ is tangent to the curves $(-\e, \e) \times \{x\}$, with $x \in \tilde \H_p$.
Since
\[
	\H 
		= \bigcup_{p \in \H} \tilde \H_p,
\]
we can now use this to construct the spacetime metric $g$ and the Killing vector field $W$ in an open neighborhood of
\[
	\{0\} \times \H \subset \R \times \H.
\]
It remains to check that this is well-defined on the overlaps, i.e.\ if $q \in \tilde \H_q$ with $q \neq p$, then the metrics on
\[
	(-\e_p, \e_p) \times \tilde \H_p \cap (-\e_q, \e_q) \times \tilde \H_q
\]
coincide.
We know that the data $(\H, \s, V)$ coincide on $\tilde \H_p$ and $\tilde \H_q$.
It thus follows by Theorem \ref{thm: main uniqueness formal} that the resulting spacetimes are isometric in a neighborhood of any point on the horizon. 
This isometry is explicitly given by the null time function, which means that the isometry is the identity map
\[
	\id: (-\e_p, \e_p) \times \tilde \H_p \cap (-\e_q, \e_q) \times \tilde \H_q \to (-\e_p, \e_p) \times \tilde \H_p \cap (-\e_q, \e_q) \times \tilde \H_q.
\]
In other words, the metrics coincide and $g$ and $W$ are well-defined.
This concludes the proof.
\end{proof}

\end{sloppypar}
\end{document}